\documentclass{amsart}
\usepackage[utf8]{inputenc}
\usepackage[hidelinks]{hyperref}
\usepackage[noadjust]{cite}
\usepackage{amsthm, amsmath, amssymb, colonequals, comment, enumitem, stmaryrd}

\newtheorem{conjecture}{Conjecture}
\newtheorem{corollary}{Corollary}
\newtheorem{lemma}{Lemma}
\newtheorem{proposition}{Proposition}
\newtheorem{theorem}{Theorem}

\theoremstyle{definition}
\newtheorem{definition}{Definition}

\theoremstyle{remark}
\newtheorem{remark}{Remark}

\DeclareMathOperator{\ad}{ad}
\DeclareMathOperator{\Aut}{Aut}
\DeclareMathOperator{\Der}{Der}
\DeclareMathOperator{\End}{End}
\DeclareMathOperator{\DC}{DC}
\DeclareMathOperator{\HDC}{HDC}
\DeclareMathOperator{\Hom}{Hom}
\DeclareMathOperator{\JC}{JC}
\DeclareMathOperator{\InnDer}{InnDer}
\DeclareMathOperator{\id}{id}
\DeclareMathOperator{\N}{\mathbb{N}}
\DeclareMathOperator{\Q}{\mathbb{Q}}
\DeclareMathOperator{\Z}{\mathbb{Z}}

\begin{document}

\title{The higher-order hom-associative Weyl algebras}
\author{Per B\"ack}
\address{Department of Business and Mathematics, M\"alar\-dalen  University,  SE-721  23  V\"aster\r{a}s, Sweden}
\email{per.back@mdu.se}


\begin{abstract}
We show that the higher-order Weyl algebras over a field of characteristic zero, which are formally rigid as associative algebras, can be formally deformed in a nontrivial way as hom-associative algebras. We also show that these hom-associative Weyl algebras arise naturally as hom-associative iterated differential polynomial rings, that they contain no zero divisors, are power-associative only when associative, and that they are simple. We then determine their commuters, nuclei, centers, and derivations. Last, we classify all hom-associative Weyl algebras up to isomorphism and conjecture that all nonzero homomorphisms between any two isomorphic hom-associative Weyl algebras are isomorphisms. The latter conjecture turns out to be stably equivalent to the Dixmier Conjecture, and hence also to the Jacobian Conjecture. 
\end{abstract}
\maketitle

\section{Introduction}
\emph{Hom-associative algebras} were introduced by Makhlouf and Silvestrov \cite{MS08} as nonassociative algebras where the associativity condition is twisted by a linear map. In particular, any associative algebra may be seen as a hom-associative algebra with twisting map the identity map. The motivation for introducing these algebras was to construct, in a most natural way, \emph{hom-Lie algebras}. The latter algebras were in turn introduced by Hartwig, Larsson, and Silvestrov \cite{HLS06} as generalizations of Lie algebras, the Jacobi identity now twisted by a linear map. In particular, any Lie algebra may be seen as a hom-Lie algebra where the twisting map is the identity map. Now, any hom-associative algebra gives rise to a hom-Lie algebra when the commutator is used as a hom-Lie bracket, that is, as a new nonassociative multiplication. When the twisting map is the identity map, the above construction is the well-known construction of a Lie algebra from an associative algebra, when the commutator is used as a Lie bracket.

The \emph{higher-order Weyl algebras} are associative algebras that can be exhibited as \emph{iterated differential polynomial rings}, the latter a special type of \emph{Ore extension}, or \emph{noncommutative polynomial ring}, as they were first called by Ore \cite{Ore33} who introduced them. Since their introduction, Ore extensions have been well studied (see e.g. \cite{GW04, Lam01, MR01, Row88} for introductions to the subject), and recently several authors have started studying various nonassociative versions (see e.g. \cite{AB26, AB25, BLOR24, BR24b, NOR18, NOR19}) and hom-associative generalizations (see e.g. \cite{Bac22, BR23, BRS18}) of them. Moreover, in \cite{BR20} the authors studied hom-associative generalizations of the first Weyl algebras in characteristic zero, and in \cite{BR22} in prime characteristic. In \cite{BR24b}, the authors introduced hom-associative generalizations of the higher-order Weyl algebras in characteristic zero and showed that families of these algebras have all their one-sided ideals as principal. 

The higher-order Weyl algebras are formally rigid as associative algebras in characteristic zero, this in the classical sense of Gerstenhaber who introduced formal deformation theory for associative algebras and rings in the famous article \cite{Ger64}. In the present article, we show that the \emph{higher-order hom-associative Weyl algebras} are formal hom-associative deformations of the higher-order Weyl algebras in characteristic zero, a result which was proven to hold for the first (hom-associative) Weyl algebra in \cite{BR20}. Thus, the higher-order Weyl algebras can indeed be deformed in characteristic zero when seen as hom-associative algebras with twisting map the identity map. In this article, we also show that the higher-order hom-associative Weyl algebras arise, similarly to their associative counterparts, as hom-associative analogues of iterated differential polynomial rings. We then prove that they contain no zero divisors, are power-associative only when associative, and that they are simple. The latter generalizes results in \cite{BR20, BR24}. We then determine their commuters, nuclei, centers, and derivations, generalizing results about the first hom-associative Weyl algebra in characteristic zero to the higher-order analogues. Last, we classify all higher-order hom-associative Weyl algebras up to isomorphism and conjecture that all nonzero homomorphisms between any two isomorphic higher-order hom-associative Weyl algebras are isomorphisms. This conjecture is known to hold for the first purely hom-associative Weyl algebra \cite{BR20}, and in a recent preprint, Zheglov~\cite{Zhe24} claims that it also holds for the first Weyl algebra. The above general conjecture turns out to be stably equivalent to the \emph{Dixmier Conjecture}, which states that any algebra endomorphism on a higher-order Weyl algebra is, in fact, an algebra automorphism. Tsuchimoto \cite{Tsu05} and Kanel-Belov and Kontsevich \cite{KBK07} have independently proven that the latter conjecture is stably equivalent to the famous \emph{Jacobian Conjecture}.

The article is organized as follows: 

In \autoref{sec:prel}, we provide preliminaries on nonassociative algebras (\autoref{subsec:non-assoc}), hom-associative algebras (\autoref{subsec:hom}), and iterated differential polynomial rings and the higher-order Weyl algebras (\autoref{subsec:iterated}).

In \autoref{sec:hom-diff}, we describe how to construct, in a natural way, iterated hom-associative differential polynomial rings from associative analogues (\autoref{prop:yau-twisted-Ore-extension}). 

In \autoref{sec:hom-Weyl}, we define hom-associative analogues of the higher-order Weyl algebras (\autoref{def:hom-weyl}). We show that these higher-order hom-associative Weyl algebras contain no zero divisors, are power-associative only if they are associative (\autoref{thm:hom-Weyl-properties}), and that they are simple (\autoref{thm:simple}). We then determine their commuters, nuclei, and centers (\autoref{thm:hom-Weyl-properties2}), as well as their derivations (\autoref{thm:derivations}). Last, we classify them up to isomorphism (\autoref{thm:hom-Weyl-isomorphism}) and conjecture that all nonzero homomorphisms between any two isomorphic hom-associative Weyl algebras are isomorphisms (\autoref{conj:hom-Dixmier}). We then show that the stable version of the above conjecture is equivalent to the Dixmier Conjecture (\autoref{prop:conj}).

In \autoref{sec:multi-param}, we recall what multi-parameter formal deformations of hom-as\-so\-cia\-ti\-ve algebras (\autoref{def:hom-assoc-deform}) and hom-Lie algebras (\autoref{def:hom-Lie-deform}) are. We show that the higher-order hom-associative Weyl algebras are multi-parameter formal deformations of the higher-order Weyl algebras (\autoref{thm:hom-Weyl-deform}), and that they induce multi-parameter formal deformations of the corresponding Lie algebras into hom-Lie algebras, when the commutator is used as a hom-Lie bracket (\autoref{cor:hom-Weyl-Lie-deform}).

\section{Preliminaries}\label{sec:prel}
Throughout this article, we denote by $\N$ the set of nonnegative integers, and by $\N_{>0}$ the set of positive integers.

\subsection{Nonassociative algebras}\label{subsec:non-assoc}
Let $R$ be an associative, commutative, and unital ring. By a \emph{nonassociative $R$-algebra}, we mean an algebra over $R$ which is not necessarily associative, and not necessarily unital. In particular, a \emph{nonassociative ring} is a nonassociative $\mathbb{Z}$-algebra. If $A$ is a nonassociative and unital $R$-algebra, then any $R$-algebra endomorphism on $A$ is assumed to respect the identity element. 

If $A$ is a nonassociative $R$-algebra, recall that a nonzero element $a\in A$ is called a \emph{left zero divisor} of $A$ if there is a nonzero $b\in A$ such that $ab=0$. Similarly, $a$ is called a \emph{right zero divisor} of $A$ if there is a nonzero $b\in A$ such that $ba=0$. An element that is a left or a right zero divisor of $A$ is simply called a \emph{zero divisor} of $A$. We denote by $D_l(A)$ the set of left zero divisors of $A$, and by $D_r(A)$ the set of right zero divisors of $A$.

An \emph{ideal} $I$ of $A$ is an additive subgroup of $A$ invariant under left and right multiplication, meaning $RI, IR, AI, IA\subseteq I$. If $\{0\}$, called the \emph{zero ideal}, and $A$ are the only ideals of $A$, then $A$ is \emph{simple}.

The \emph{commutator} of a nonassociative $R$-algebra $A$ is the $R$-bilinear map $[\cdot,\cdot]\colon A\times A\to A$ defined by $[a,b]\colonequals ab-ba$ for any $a,b\in A$. The \emph{commuter} of $A$, written $C(A)$, is defined as $\{a\in A\mid [a,b]=0\text{ for any } b\in A\}$. The \emph{associator} of $A$ is the $R$-trilinear map $(\cdot,\cdot,\cdot)\colon A\times A\times A\to A$ defined by $(a,b,c)\colonequals (ab)c-a(bc)$ for any $a,b,c\in A$. $A$ is called \emph{power associative} if $(a,a,a)=0$, \emph{left alternative} if $(a,a,b)=0$, \emph{right alternative} if $(a,b,b)=0$, flexible if $(a,b,a)=0$ for any $a,b\in A$. Moreover, $A$ is associative precisely when $(a,b,c)=0$ for any $a,b,c\in A$.

The \emph{left, middle,} and \emph{right nucleus} of $A$ are defined, respectively, as follows: $N_l(A)\colonequals\{a\in A\mid (a,b,c)=0\text { for any } b,c\in A\}$, $N_m(A)\colonequals\{b\in A\mid (a,b,c)=0\text { for any } a,c\in A\}$, and $N_r(A)\colonequals\{c\in A\mid (a,b,c)=0\text { for any } a,b\in A\}$. It turns out that $N_l(A), N_m(A),$ and $N_r(A)$ are all associative subalgebras of $A$. The \emph{nucleus} of $A$, denoted by $N(A)$, is defined as $N_l(A)\cap N_m(A)\cap N_r(A)$. The \emph{center} of $A$, written $Z(A)$ is defined as the associative and commutative subalgebra $C(A)\cap N(A)$ of $A$. A \emph{derivation} on $A$ is an $R$-linear map $\delta\colon A\to A$ satisfying, for all $a,b\in A$, the identity $\delta(ab)=\delta(a)b+a\delta(b)$. The set of all derivations on $A$ is denoted by $\Der_R(A)$. Now, if $A$ is associative and $a$ is an arbitrary element of $A$, then $\ad_a\colonequals[a,\cdot]\colon A\to A$ is a derivation, called an \emph{inner derivation} (If $A$ is not associative, such a map need not be a derivation, however). Whenever $A$ is associative, we denote by $\InnDer_R(A)$ the set of all inner derivations on $A$.

Suppose that $A$ is a nonassociative $\Q$-algebra. A map $\varphi\colon A\to A$ is said to be \emph{locally nilpotent} if for each $a\in A$, there exists an $n\in\N_{>0}$ such that $\varphi^n(a)=0$. If $\varphi$ is locally nilpotent, then we define $e^\varphi$ as the formal power series $\sum_{i\in\N}\frac{\varphi^i}{i!}$ where $\varphi^0\colonequals\id_A$. The next proposition is a standard result on such maps. We provide a proof for the convenience of the reader.

\begin{proposition}\label{prop:nilpotent-maps}
If $\varphi_1,\ldots,\varphi_n$ are pairwise commuting locally nilpotent maps on a nonassociative $\Q$-algebra, then $e^{\varphi_1}\cdots e^{\varphi_n}=e^{\varphi_1+\cdots+\varphi_n}$.
\end{proposition}

\begin{proof}
By e.g. using the binomial theorem, one can see that the sum of two pairwise commuting locally nilpotent maps is again a locally nilpotent map. By induction, any finite sum of pairwise commuting, locally nilpotent maps is also a locally nilpotent map, so the statement in \autoref{prop:nilpotent-maps} makes sense. We now prove this statement, which we denote by P$(n)$, by induction on $n$. If $\varphi$ is a locally nilpotent map, we let $\binom{i}{j}\varphi^{i-j}\colonequals0$ whenever $i<j$. Using this, we see that P$(2)$ holds, since for any two commuting locally nilpotent maps $\varphi_1$ and $\varphi_2$,
\begin{align*}
e^{\varphi_1+\varphi_2}&\colonequals\sum_{i\in\N}\frac{(\varphi_1+\varphi_2)^i}{i!}=\sum_{i\in\N}\sum_{j=0}^i\binom{i}{j}\frac{\varphi_1^j\varphi_2^{i-j}}{i!}=\sum_{i\in\N}\sum_{j\in\N}\binom{i}{j}\frac{\varphi_1^j\varphi_2^{i-j}}{i!}\\
&\phantom{:}=\sum_{i\in\N}\sum_{j\in\N}\left.\left.\frac{\varphi_1^j\varphi_2^{i-j}}{j!(i-j)!}\right[\ell\colonequals i-j\right]=\sum_{j\in\N}\frac{\varphi_1^j}{j!}\sum_{\ell\in\N}\frac{\varphi_2^\ell}{\ell !}\equalscolon e^{\varphi_1}e^{\varphi_2}.
\end{align*}
Let $\varphi_1,\ldots,\varphi_{n+1}$ be pairwise commuting, locally nilpotent maps, and assume that P$(n)$ holds. Then P($n+1$) holds as well, since
\begin{align*}
e^{\varphi_1+\cdots+\varphi_{n+1}}&\stackrel{\phantom{\text{P}(n)}}{=}\sum_{i\in\N}\frac{(\varphi_1+\cdots + \varphi_{n+1})^i}{i!}=\sum_{i\in\N}\sum_{j=0}^i\binom{i}{j}\frac{(\varphi_1+\cdots+\varphi_n)^j\varphi_{n+1}^{i-j}}{i!}\\
&\stackrel{\phantom{\text{P}(n)}}{=}\sum_{i\in\N}\sum_{j=0}^i\left.\left.\frac{(\varphi_1+\cdots+\varphi_n)^j\varphi_{n+1}^{i-j}}{j!(i-j)!}\right[\ell\colonequals i-j\right]\\
&\stackrel{\phantom{\text{P}(n)}}{=}\sum_{j\in\N}\frac{(\varphi_1+\cdots+\varphi_n)^j}{j!}\sum_{\ell\in\N}\frac{\varphi_{n+1}^\ell}{\ell !}\equalscolon e^{\varphi_1+\cdots +\varphi_n}e^{\varphi_{n+1}}\\
&\stackrel{\text{P}(n)}{=}e^{\varphi_1}\cdots e^{\varphi_n}e^{\varphi_{n+1}}.\qedhere
\end{align*}
\end{proof}

\subsection{Hom-associative algebras and hom-Lie algebras}\label{subsec:hom}
We recall the definition of a \emph{hom-associative algebra}, a notion first introduced by Makhlouf and Silvestrov \cite{MS08}.

\begin{definition}[Hom-associative algebra]
A \emph{hom-associative algebra} over an associative, commutative, and unital ring $R$ is a nonassociative $R$-algebra $A$ equipped with an $R$-linear map $\alpha\colon A\to A$, called \emph{twisting map}, satisfying  for all $a,b,c\in A$ the following identity:
\[\alpha(a)(bc)=(ab)\alpha(c)\quad \text{(hom-associativity)}.\]
\end{definition}

\begin{remark}
From the hom-associative identity we see that we may view any associative $R$-algebra $A$ as a hom-associative $R$-algebra with twisting map the identity map $\id_A\colon A\to A$. Similarly, any nonassociative $R$-algebra $A$ may be viewed as a hom-associative $R$-algebra with twisting map the zero map $0_A\colon A\to A$.
\end{remark}

\begin{definition}[Hom-associative ring]
A \emph{hom-associative ring} is a hom-associative algebra over $\Z$.
\end{definition}

If $A$ and $B$ are hom-associative $R$-algebras with twisting map $\alpha_A$ and $\alpha_B$, respectively, then a \emph{hom-associative $R$-algebra homomorphism} $\varphi\colon A\to B$ is a homomorphism of nonassociative $R$-algebras, i.e. an $R$-linear and multiplicative map, satisfying $\varphi\circ\alpha_A=\alpha_B\circ\varphi$. We denote by $\Hom_R(A,B)$
the set of all hom-associative $R$-algebra homomorphisms from $A$ to $B$, and let $\End_R(A)\colonequals \Hom_R(A,A)$ denote the corresponding set of \emph{endomorphisms}. We note that whenever $\alpha_A=\id_A$ and $\alpha_B=\id_B$, or $\alpha_A=0_A$ and $\alpha_B=0_B$, the identity $\varphi\circ\alpha_A=\alpha_B\circ\varphi$ becomes trivial, and we then have the usual definitions (notations) for (the sets of) homomorphisms and endomorphisms of associative and nonassociative $R$-algebras, respectively.

There is a now classical construction due to Yau \cite[Corollary 2.5]{Yau09} which takes as input an associative $R$-algebra with an $R$-algebra endomorphism, and gives as output a hom-associative $R$-algebra with twisting map the very same $R$-algebra endomorphism:

\begin{proposition}[Yau \cite{Yau09}]\label{prop:yau}
Let $A$ be an associative $R$-algebra with $\alpha\in\End_R(A)$. Define a new multiplication $*$ on $A$ by $a*b\colonequals \alpha(ab)$. Then the resulting nonassociative $R$-algebra, denoted by $A^\alpha$, is hom-associative with twisting map $\alpha$.
\end{proposition}

Fr{\'e}gier and Gohr \cite{FG09} realized that unitality is a much too strong concept for hom-associative algebras, and so they introduced a weaker notion of unitality called \emph{weak unitality}:

\begin{definition}[Weak unitality]
A hom-associative $R$-algebra $A$ is called \emph{weakly unital} if there is an $e\in A$, called a \emph{weak identity element}, such that $ae=ea=\alpha(a)$ hold for any $a\in A$.
\end{definition}

The next result, due to Fr{\'e}gier and Gohr \cite[Example 2.2]{FG09}, shows that any associative and unital $R$-algebra with an $R$-algebra endomorphism gives rise to a weakly unital hom-associative algebra over $R$.

\begin{corollary}[Fr{\'e}gier and Gohr \cite{FG09}]
If $A$ is an associative and unital $R$-algebra with $\alpha\in\End_R(A)$, then $A^\alpha$ is weakly unital with weak identity element $1_A$.
\end{corollary}

In general, a weak identity element does not need to be unique. However, if $\alpha$ in the above corollary is injective, then it is not too hard to see that the weak identity element $1_A$ of $A^\alpha$ is indeed unique (see \cite[Lemma 3.1]{BR22}):

\begin{lemma}[B{\"a}ck and Richter \cite{BR22}]\label{lem:unique-weak-identity}
If $A$ is an associative and unital $R$-algebra with $\alpha\in\End_R(A)$ injective, then $1_A$ is a unique weak identity element of $A^\alpha$.
\end{lemma}

The following four lemmas (see \cite[Lemma 3.2, 3.3, 3.7, and 3.4]{BR22}) relates the zero divisors, commuter, derivations, and morphisms of $A$ to those of $A^\alpha$:

\begin{lemma}[B{\"a}ck and Richter \cite{BR22}]\label{lem:zero-divisors}
Let $A$ be an associative $R$-algebra with $\alpha\in\End_R(A)$. Then $D_l(A) \subseteq D_l(A^\alpha)$ and $D_r(A)\subseteq D_r(A^\alpha)$, with equality if $\alpha$ is injective.
\end{lemma}

\begin{lemma}[B{\"a}ck and Richter \cite{BR22}]\label{lem:commuter}
Let $A$ be an associative $R$-algebra with $\alpha\in\End_R(A)$. Then $C(A)\subseteq C(A^\alpha)$, with equality if $\alpha$ is injective.
\end{lemma}

\begin{lemma}[B{\"a}ck and Richter \cite{BR22}]\label{lem:derivations}
Let $A$ be an associative and unital $R$-algebra with $\alpha\in\End_R(A)$ injective. Then $\Der_R(A^\alpha) =\{\delta\in\Der_R(A)\mid \delta\circ\alpha = \alpha\circ\delta\}$ if and only if $\delta'(1_A) = 0$ for any $\delta'\in\Der_R(A^\alpha)$.
\end{lemma}

\begin{lemma}[B{\"a}ck and Richter \cite{BR22}]\label{lem:morphisms}
Let $A$ be an associative R-algebra with $\alpha, \alpha'\in \End_R(A)$. Then $\{\varphi\in\End_R(A)\mid \varphi\circ \alpha = \alpha'\circ \varphi \}\subseteq \Hom_R(A^\alpha,A^{\alpha'})$, with equality if $\alpha'$ is injective.
\end{lemma}

Recall that an associative $R$-algebra becomes a Lie algebra over $R$ when the commutator is used as a Lie bracket. The motivation for introducing hom-associative algebras in the first place was to find a counterpart to associative algebras when replacing Lie algebras by so-called \emph{hom-Lie algebras} in the above construction, the latter first introduced by Hartwig, Larsson, and Silvestrov \cite{HLS06}.

\begin{definition}[Hom-Lie algebra]
A \emph{hom-Lie algebra} over an associative, commutative, and unital ring $R$ is a nonassociative $R$-algebra $L$ with multiplication $[\cdot,\cdot]_L\colon L\times L\to L$, called the \emph{hom-Lie bracket}, equipped with an $R$-linear map $\alpha\colon L\to L$, called \emph{twisting map}, satisfying for all $a,b,c\in L$ the following identities:
\begin{align*}
[a,a]_L&=0\quad \text{(alternativity)},\\
[\alpha(a),[b,c]_L]_L+[\alpha(c),[a,b]_L]_L+[\alpha(b),[c,a]_L]_L&=0\quad \text{(hom--Jacobi identity).}
\end{align*}
\end{definition}

\begin{remark}
A hom-Lie algebra with twisting map the identity map is a Lie algebra.
\end{remark}

The following result is due to Makhlouf and Silvestrov \cite[Proposition 1.6]{MS08}.

\begin{proposition}[Makhlouf and Silvestrov \cite{MS08}]\label{prop:hom-Lie-commutator}
If $A$ is a hom-associative $R$-algebra with twisting map $\alpha$, then $A$ becomes a hom-Lie algebra over $R$ with twisting map $\alpha$ when the commutator is used as a hom-Lie bracket.
\end{proposition}

\subsection{Iterated differential polynomial rings and the higher-order Weyl algebras}\label{subsec:iterated}
Given an associative and unital ring $R$, in 1933 Ore \cite{Ore33} introduced what he called a ``noncommutative polynomial ring'' over $R$. As a left $R$-module, this ring is equal to $R[x]$, but is equipped with a noncommutative multiplication subject only to the relation $xR\subseteq Rx+R$. One can show that such a multiplication exists precisely when there is a ring endomorphism $\sigma\colon R\to R$ and a \emph{$\sigma$-derivation} $\delta\colon R\to R$, that is, an additive map satisfying, for any $r,s\in R$, 
\[\delta(rs)=\sigma(r)\delta(s)+\delta(r)s.\]
The multiplication is then uniquely defined by the relation
\begin{equation}
xr=\sigma(r)x+\delta(r).\label{eq:ore-mult}
\end{equation}
The left $R$-module $R[x]$ equipped with the above multiplication is an associative and unital ring denoted by $R[x;\sigma,\delta]$ and called an \emph{Ore extension} of $R$. From \eqref{eq:ore-mult}, one can deduce that the multiplication in $R[x;\sigma,\delta]$ is given on monomials (and then extended bi-additively to the whole of $R[x;\sigma,\delta]$) by
\begin{equation}
(rx^i)(sx^j)=\sum_{\ell\in\mathbb{N}}r\pi_\ell^i(s)x^{j+\ell},\label{eq:ore-mult2}
\end{equation}
for any $r,s\in R$ and $i,j\in\N$ where $\pi_\ell^i\colon R\to R$ is defined as the sum of all $\binom{i}{\ell}$ compositions of $\ell$ instances of $\sigma$ and $i-\ell$ instances of $\delta$. For instance, $\pi_2^3\colonequals\sigma\circ\sigma\circ\delta + \sigma\circ\delta\circ\sigma + \delta\circ\sigma\circ\sigma$ while $\pi^0_0$ is defined as $\id_R$. Whenever $\ell>i$, we set $\pi_\ell^i\colonequals 0_R$. 

If $\sigma=\id_R$, then $R[x;\id_R,\delta]$, written $R[x;\delta]$, is called a \emph{differential polynomial ring} of $R$, and $\delta$ is simply a derivation (see \autoref{subsec:non-assoc}). If in addition $\delta=0_R$, then $R[x;0_R]$ is just the ordinary polynomial ring $R[x]$. If $\delta_1,\ldots,\delta_n$ are pairwise commuting derivations on $R$, then we may construct an \emph{iterated differential polynomial ring} of $R$ as follows (see e.g.  \cite[Exercise 2H]{GW04}). First, we set $S_1\colonequals R[x_1;\delta_1]$. Then $\delta_2$ extends uniquely to a derivation $\widehat{\delta}_2$ on $S_1$ such that $\widehat\delta_2(x_1)=0$. Next, we set $S_2\colonequals S_1[x_2;\widehat{\delta}_2]$. Once $S_k$ has been constructed for some $\ell<n$, we define $S_{\ell+1}\colonequals S_k[x_{\ell+1};\widehat{\delta}_{\ell+1}]$ where $\widehat{\delta}_{\ell+1}$ is the unique derivation on $S_\ell$ such that $\widehat\delta_{\ell+1}\vert_R=\delta_{\ell+1}$ and $\widehat\delta_{\ell+1}(x_j)=0$ for $1\leq j\leq \ell$. We may now construct an iterated differential polynomial ring $R[x_1;\delta_1][x_2;\widehat\delta_2]\cdots[x_n;\widehat{\delta}_n]$, which we denote by $R[x_1,\ldots,x_n;\delta_1,\ldots,\delta_n]$. $R[x_1,\ldots,x_n;\delta_1,\ldots,\delta_n]$ is moreover a free left $R$-module with basis all monomials $x_1^{i_1}\cdots x_n^{i_n}$ where $i_1,\ldots,i_n\in\N$. The multiplication in this ring is given by the bi-additive extension of the relations
\begin{align}
&\big(rx_1^{i_1}\cdots x_n^{i_n}\big) \big(sx_1^{j_1}\cdots x_n^{j_n}\big)\nonumber\\&=\sum_{\ell_1=0}^{i_1}\cdots\sum_{\ell_n=0}^{i_n}\binom{i_1}{\ell_1}\cdots\binom{i_n}{\ell_n}r\delta_1^{i_1-\ell_1}\circ \cdots \circ\delta_n^{i_n-\ell_n}(s)x_1^{j_1+\ell_1}\cdots x_n^{j_n+\ell_n}\label{eq:multi-multiplication}
\end{align}
for any $r,s\in R$ and $i_1,\ldots,i_n,j_1,\ldots,j_n\in\N$.

The \emph{$n$th Weyl algebra} $A_n$, where $n\in\N_{>0}$, over a field $K$ of characteristic zero is the free, associative, and unital $K$-algebra on $2n$ letters $x_1,\ldots, x_n, y_1,\ldots, y_n$ modulo the commutation relations $[x_i,x_j]=[y_i,y_j]=0$ and $[x_i,y_j]=\delta_{ij}$ for $1\leq i,j\leq n$. Here, $\delta_{ij}$ denotes the \emph{Kronecker delta} on $A_n$,
\[\delta_{ij}\colonequals \begin{cases}1_{A_n}&\text{if } i=j,\\0&\text{otherwise.}\end{cases}\]

The Weyl algebras $A_1,\ldots,A_n$ are naturally ordered by inclusion, $A_1\subsetneq \cdots \subsetneq A_n$, and so we also refer to them as \emph{higher-order Weyl algebras}. The higher-order Weyl algebras show up in many different contexts and under many different guises (see e.g. the excellent survey article by Coutinho \cite{Cou97}). They are probably best known as algebras of quantum mechanical operators, where each $x_i$ plays the role of a momentum operator, and each $y_i$ the role of a position operator, as well as for their relation to the famous \emph{Jacobian Conjecture} described here below.

Denote by $\partial /\partial x_1,\ldots,\partial/\partial x_n, \partial/\partial y_1,\ldots, \partial/\partial y_n$ the formal partial derivatives with respect to $x_1,\ldots,x_n, y_1,\ldots,y_n$ on $A_n$. It is not too hard to see that $\ad_{x_\ell}=\partial/\partial y_\ell$ and $\ad_{y_\ell}=-\partial/\partial x_\ell$ for $1\leq \ell\leq n$. The ordinary polynomial ring $R\colonequals K[y_1,\ldots,y_n]$ in $n$ indeterminates over $K$ is (isomorphic to) a subring of $A_n$. If we restrict $\partial/\partial y_1,\ldots, \partial/\partial y_n$ to $R$, then $\partial/\partial y_1,\ldots, \partial/\partial y_n$ are pairwise commuting derivations on $R$. Hence we may form the iterated differential polynomial ring (see \autoref{subsec:iterated}), or indeed a $K$-algebra, $R[x_1,\ldots,x_n;\partial/\partial y_1,\ldots \partial/\partial y_n]$. As it turns out, $A_n$ is isomorphic to the above $K$-algebra, with the isomorphism induced by the canonical assignment $x_i\mapsto x_i$, $y_i\mapsto y_i$ (see e.g. \cite[Exercise 2J]{GW04}).

As a $K$-vector space, $A_n$ has a basis consisting of monomials $y_1^{i_1}\cdots y_n^{i_n}x_1^{j_1}\cdots x_n^{j_n}$, $i_1,\dots,i_n, j_1,\dots, j_n\in\N$, and the identity element $1_{A_n}$ is $1_Ky_1^0\cdots y_n^0x_1^0\cdots x_n^0$. An element of $A_n$ is thus a finite $K$-linear combination $p(y_1,\ldots,y_n,x_1,\ldots,x_n)$ of such monomials, i.e.
 a polynomial. It is not too hard to prove that $A_n$ contains no zero divisors and has a center equal to $K$. Since $A_n$ contains no zero divisors, any nonzero, $K$-linear, and multiplicative map $\varphi$ on $A_n$ respects the identity element and is hence a $K$-algebra endomorphism;  $\varphi(1_{A_n})=\varphi(1_{A_n})\varphi(1_{A_n})\iff \varphi(1_{A_n})(1_{A_n}-\varphi(1_{A_n}))=0\iff \varphi(1_{A_n})=1_{A_n}$. Littlewood \cite[Theorem X]{Lit33} proved that $A_1$ is simple when $K=\mathbb{R}$ and when $K=\mathbb{C}$, and Hirsch \cite[Theorem]{Hir37} then showed that this result also holds for any higher-order Weyl algebra $A_n$ over a field $K$ of characteristic zero. Sridharan \cite[Remark 6.2 and Theorem 6.1]{Sri61} has shown that the Hochschild cohomology of $A_n$ vanishes in all positive degrees (see also \cite[Theorem 5]{GG14} for a different proof of this fact by Gerstenhaber and Giaquinto). The vanishing in the first degree implies that all derivations on $A_n$ are inner, and the vanishing in the second degree implies that $A_n$ is formally rigid as an associative algebra in the sense of Gerstenhaber \cite{Ger64}. 

Since $A_n$ is simple, any $K$-algebra endomorphism on $A_n$ is injective (the kernel of any $K$-algebra endomorphism is an ideal, and if $A_n$ is simple, the kernel must equal the zero ideal). In \cite[II. Probl\`{e}mes]{Dix68}, Dixmier asked whether all $K$-algebra endomorphisms on $A_1$ are also surjective, so that $\End_K(A_1)=\Aut_K(A_1)$? The conjecture that the equality $\End_K(A_1)=\Aut_K(A_1)$ holds is known as the \emph{Dixmier Conjecture for $A_1$}, and in a recent preprint, Zheglov \cite[Theorem 1.1]{Zhe24} claims that it is actually true. For a fixed $n\in\mathbb{N}_{>0}$, the \emph{Dixmier Conjecture} in general asserts that $\End_K(A_n)\backslash\{0\}=\Aut_K(A_n)$. The conjecture is still open for any $n$ greater than 1.

\begin{conjecture}[The Dixmier Conjecture ($\DC_n$)]\label{conj:Dixmier}
$\End_K(A_n)=\Aut_K(A_n)$.
\end{conjecture}

Let $\varphi\colonequals (\varphi_1,\ldots,\varphi_n)\colon K^n\to K^n$ be a \emph{polynomial map}, that is, a map of the form $(x_1,\ldots,x_n)\mapsto (\varphi_1(x_1,\ldots,x_n),\ldots,\varphi_n(x_1,\ldots,x_n))$ where $\varphi_1,\ldots\varphi_n\in K[x_1,\ldots,x_n]$, the ordinary polynomial ring in $n$ indeterminates over $K$. We say that $\varphi$ is \emph{invertible} if $\varphi$ has an inverse which is also a polynomial map. Let $J(\varphi)\colonequals (\partial \varphi_i/\partial x_j)_{1\leq i,j\leq n}$, the \emph{Jacobian matrix}. 
The \emph{Jacobian Conjecture} is a famous conjecture in algebraic geometry, which reads as follows:

\begin{conjecture}[The Jacobian Conjecture ($\JC_n$)]\label{conj:Jacobian}If $\varphi\colon K^{n}\to K^{n}$ is a polynomial map with $\det(J(\varphi))\in K\backslash\{0\}$, then $\varphi$ is invertible.
\end{conjecture}

It is well known that $\DC_n$ implies $\JC_n$. Tsuchimoto \cite[Corollary 7.3]{Tsu05} and Kanel-Belov and Kontsevich \cite[Theorem 1]{KBK07} have independently proven that $\JC_{2n}$ implies $\DC_{n}$ (note that $\JC_n$ is true when $n=1$). The conjunction of the conjectures $\DC_n$ for all $n\in\N_{>0}$ is known as the \emph{Stable Dixmier Conjecture}, and it is denoted by $\DC_\infty$. Similarly, the conjunction of the conjectures $\JC_n$ for all $n\in\N_{>0}$ is known as the \emph{Stable Jacobian Conjecture}, and it is denoted by $\JC_\infty$. Hence $\DC_\infty$ and $\JC_\infty$ are equivalent, or put differently, $\DC_n$ and $\JC_n$ are \emph{stably equivalent}.

\section{Hom-associative iterated differential polynomial rings}\label{sec:hom-diff}
Let $R$ be a nonassociative ring with additive maps $\sigma,\delta\colon R\to R$. The \emph{nonassociative Ore extension} of $R$, denoted by $R[x;\sigma,\delta]$, is the set of formal sums $\sum_{i\in\mathbb{N}}r_i x^i$ where only finitely many $r_i\in R$ are nonzero, equipped with pointwise addition and multiplication defined by the bi-additive extension of the relations \eqref{eq:ore-mult2}. Now, if $R$ is hom-associative with twisting map $\alpha$, we would like to extend, in a natural way, $\alpha$ to a map $\widehat\alpha$ on $R[x;\sigma,\delta]$, such that $R[x;\sigma,\delta]$ becomes hom-associative with twisting map $\widehat\alpha$. In particular, if $\alpha$ and $\sigma$ are ring endomorphisms and $\delta$ is a $\sigma$-derivation on an associative and unital ring $R$, then we would like to extend $\alpha$ to a ring endomorphism $\widehat\alpha$ on $R[x;\sigma,\delta]$ such that $R^\alpha[x;\sigma,\delta]=R[x;\sigma,\delta]^{\widehat\alpha}$ (see \autoref{prop:yau} in \autoref{subsec:hom}). The next result, which is a slightly improved version of \cite[Corollary 3.2]{BRS18}, provides a way to extend $\alpha$ in this way.

\begin{lemma}\label{lem:homogeneous-extension}
Let $R$ be an associative and unital ring with ring endomorphisms $\alpha$ and $\sigma$, and a $\sigma$-derivation $\delta$. Then $\alpha$ extends uniquely to a ring endomorphism $\widehat\alpha$ on $R[x;\sigma,\delta]$ such that $\widehat\alpha(x)=x$ if and only if $\alpha$ commutes with $\sigma$ and $\delta$.
\end{lemma}

\begin{proof}
Assume that $\alpha$ is a ring endomorphism on $R$ that extends uniquely to a ring endomorphism $\widehat\alpha$ on $S\colonequals R[x;\sigma,\delta]$ such that $\widehat\alpha(x)=x$. Then $\alpha(\sigma(r))x+\alpha(\delta(r))=\widehat\alpha(\sigma(r)x+\delta(r))=\widehat\alpha(xr)=\widehat\alpha(x)\alpha(r)=x\alpha(r)=\sigma(\alpha(r))x+\delta(\alpha(r))$ for any $r\in R$. By comparing coefficients, $\alpha$ commutes with $\sigma$ and $\delta$.

Now assume that $\alpha$ is a ring endomorphism on $R$ that commutes with $\sigma$ and $\delta$. We show that $\alpha$ extends uniquely to a ring endomorphism $\widehat\alpha$ on $S$ such that $\widehat\alpha(x)=x$. By necessity, for any $r\in R$ and $i\in\N$, $\widehat\alpha(rx^i)\colonequals\alpha(r)\widehat\alpha(x)^i=\alpha(r)x^i$ where $\widehat\alpha$ is extended additively to the whole of $S$. It thus suffices to show that $\widehat\alpha\left(rx^isx^j\right)=\widehat\alpha\left(rx^i\right)\widehat\alpha\left(sx^j\right)$ for any $r,s\in R$ and $i,j\in\N$. Since $\alpha$ commutes with $\sigma$ and $\delta$, it also commutes with $\pi_\ell^i$ for any $i,\ell\in\N$. Hence
\begin{align*}
\widehat\alpha\left(rx^isx^j\right)&=\widehat\alpha\left(\sum_{\ell\in\N}r\pi_k^i(s)x^{j+\ell}\right)=\sum_{\ell\in\N}\widehat\alpha\left(r\pi_\ell^i(s)x^{j+\ell}\right)=\sum_{\ell\in\N}\alpha(r)\alpha(\pi_\ell^i(s))x^{j+\ell}\\&=\sum_{\ell\in\N}\alpha(r)\pi_\ell^i(\alpha(s))x^{j+\ell}= \alpha(r)x^i\alpha(s)x^j=\widehat\alpha\left(rx^i\right)\widehat\alpha\left(sx^j\right).
\end{align*}

That $\widehat\alpha$ respects the identity element follows from $\alpha$ respecting the identity element.
\end{proof}

We say that $\alpha$ in \autoref{lem:homogeneous-extension} is \emph{extended homogeneously} to $\widehat\alpha$. With this notation, we have the following result (see \cite[Proposition 7]{Bac22}), which is a slightly improved version of \cite[Proposition 5.5]{BRS18}:

\begin{proposition}[B{\"a}ck \cite{Bac22}]\label{prop:hom*ore}
Let $R$ be an associative and unital ring with ring endomorphisms $\alpha$ and $\sigma$, and a $\sigma$-derivation $\delta$. If $\alpha$ is extended homogeneously to a ring endomorphism $\widehat\alpha$ on $R[x;\sigma,\delta]$, then $R[x;\sigma,\delta]^{\widehat\alpha}=R^\alpha[x;\sigma,\delta]$.
\end{proposition}

In order to construct hom-associative versions of the higher-order Weyl algebras, we would like to apply the same reasoning as above to iterated differential polynomial rings. 

\begin{lemma}\label{lem:iterated-homogeneous-extension}
Let $R$ be an associative and unital ring with a ring endomorphism $\alpha$ and pairwise commuting derivations $\delta_1,\ldots,\delta_n$. Then $\alpha$ extends uniquely to a ring endomorphism $\widehat\alpha$ on $R[x_1,\ldots,x_n;\delta_1,\ldots,\delta_n]$ such that $\widehat\alpha(x_\ell)=x_\ell$ for $1\leq \ell\leq n$ if and only if $\alpha$ commutes with $\delta_1,\ldots,\delta_n$.
\end{lemma}

\begin{proof}
Assume that $\alpha$ is a ring endomorphism on $R$ that extends uniquely to a ring endomorphism $\widehat\alpha$ on $S\colonequals R[x_1,\ldots,x_n;\delta_1,\ldots,\delta_n]$ such that $\widehat\alpha(x_\ell)=x_\ell$ for $1\leq \ell\leq n$. Then $\alpha(r)x_\ell+\alpha(\delta_\ell(r))=\widehat\alpha(rx_\ell+\delta_\ell(r))=\widehat\alpha(x_\ell r)=\widehat\alpha(x_\ell)\alpha(r)=x_\ell\alpha(r)=\alpha(r)x_\ell+\delta_\ell(\alpha(r))$ for any $r\in R$ and $1\leq \ell\leq n$. By comparing coefficients, $\alpha$ commutes with $\delta_\ell$.

Now assume that $\alpha$ is a ring endomorphism on $R$ that commutes with $\delta_1,\ldots,\delta_n$. We show that $\alpha$ extends uniquely to a ring endomorphism $\widehat\alpha$ on $S$ such that $\widehat\alpha(x_\ell)=x_\ell$ for $1\leq \ell \leq n$. By necessity, for any $r\in R$ and $i_1,\ldots,i_n\in\N$, $\widehat\alpha(rx_1^{i_1}\cdots x_n^{i_n})\colonequals\alpha(r)\widehat\alpha(x_1)^{i_1}\cdots\alpha(x_n)^{i_n}=\alpha(r)x_1^{i_1}\cdots x_n^{i_n}$ where $\widehat\alpha$ is extended additively to the whole of $S$. We now show that for any $r,s\in R$ and $i_1,\ldots,i_n, j_1,\ldots,j_n\in\N$, the equality $\widehat\alpha\big(\big(rx_1^{i_1}\cdots x_n^{i_n}\big)\big( sx_1^{j_1}\cdots x_n^{j_n}\big)\big)= \widehat\alpha\big(rx_1^{i_1}\cdots x_n^{i_n}\big)\widehat\alpha\big(sx_1^{j_1}\cdots x_n^{j_n}\big)$ holds. To this end, we use induction on the total degree $I\colonequals i_1+\cdots + i_n\in\N$ of the leftmost monomial (we exclude the case $r=0$ since then the equality holds trivially). The base case $I=0$ is immediate since $\widehat\alpha(rsx_1^{j_1}\cdots x_n^{j_n})=\alpha(rs)\widehat\alpha(x_1^{j_1}\cdots x_n^{j_n})=\alpha(r)\alpha(s)\widehat\alpha(x_1^{j_1}\cdots x_n^{j_n})=\alpha(r)\widehat\alpha(sx_1^{j_1}\cdots x_n^{j_n})$. Note that the induction case $I+1$ implies that there is some $\ell$ for which $x_\ell$ has degree $i_\ell+1$. Since $x_1^{i_1},\ldots,x_n^{i_n}$ all commute,
\begin{align*}
 &\widehat\alpha\big(\big(rx_1^{i_1}\cdots x_\ell^{i_\ell+1}\cdots x_n^{i_n}\big)\big(sx_1^{j_1}\cdots x_n^{j_n}\big)\big)=\widehat\alpha\big(rx_1^{i_1}\cdots x_n^{i_n} (x_\ell s)x_1^{j_1}\cdots x_n^{j_n}\big)\\
 &=\widehat\alpha\big(rx_1^{i_1}\cdots x_n^{i_n} (sx_\ell+\delta_\ell(s))x_1^{j_1}\cdots x_n^{j_n}\big)\\
 &=\widehat\alpha\big(\big(rx_1^{i_1}\cdots x_n^{i_n}\big)\big(sx_1^{j_1}\cdots x_\ell^{j_\ell+1}\cdots x_n^{j_n}\big)\big) + \widehat\alpha\big(\big(rx_1^{i_1}\cdots x_n^{i_n}\big)\big(\delta_\ell(s)x_1^{j_1}\cdots x_n^{j_n}\big)\big)\\
 &\stackrel{I}{=}\widehat\alpha\big(rx_1^{i_1}\cdots x_n^{i_n}\big)\widehat\alpha\big(sx_1^{j_1}\cdots x_\ell^{j_\ell+1}\cdots x_n^{j_n}\big) + \widehat\alpha\big(rx_1^{i_1}\cdots x_n^{i_n}\big)\widehat\alpha\big(\delta_\ell(s)x_1^{j_1}\cdots x_n^{j_n}\big)\\
 &=\widehat\alpha\big(rx_1^{i_1}\cdots x_n^{i_n}\big)\alpha(s)x_1^{j_1}\cdots x_\ell^{j_k+1}\cdots x_n^{j_n} + \widehat\alpha\big(rx_1^{i_1}\cdots x_n^{i_n}\big)\alpha(\delta_\ell(s))x_1^{j_1}\cdots x_n^{j_n}\\
 &=\widehat\alpha\big(rx_1^{i_1}\cdots x_n^{i_n}\big)\alpha(s)x_1^{j_1}\cdots x_\ell^{j_\ell+1}\cdots x_n^{j_n} + \widehat\alpha\big(rx_1^{i_1}\cdots x_n^{i_n}\big)\delta_\ell(\alpha(s))x_1^{j_1}\cdots x_n^{j_n}\\
 &=\widehat\alpha\big(rx_1^{i_1}\cdots x_n^{i_n}\big)\alpha(s)x_\ell x_1^{j_1}\cdots x_n^{j_n} + \widehat\alpha\big(rx_1^{i_1}\cdots x_n^{i_n}\big)\delta_k(\alpha(s))x_1^{j_1}\cdots x_n^{j_n}\\
 &=\widehat\alpha\big(rx_1^{i_1}\cdots x_n^{i_n}\big)(\alpha(s)x_\ell+\delta_\ell(\alpha(s)))x_1^{j_1}\cdots x_n^{j_n}=\widehat\alpha\big(rx_1^{i_1}\cdots x_n^{i_n}\big)(x_\ell\alpha(s))x_1^{j_1}\cdots x_n^{j_n}\\
 &=\widehat\alpha\big(rx_1^{i_1}\cdots x_\ell^{i_\ell+1}\cdots x_n^{i_n}\big)\widehat\alpha\big(sx_1^{j_1}\cdots x_n^{j_n}\big)
\end{align*}

That $\widehat\alpha$ respects the identity element follows from $\alpha$ respecting the identity element.
\end{proof}

\begin{proposition}\label{prop:yau-twisted-Ore-extension}
Let $R$ be an associative and unital ring with a ring endomorphism $\alpha$ and pairwise commuting derivations $\delta_1,\ldots,\delta_n$. If $\alpha$ commutes with $\delta_1,\ldots,\delta_n$ and is extended homogeneously to a ring endomorphism $\widehat\alpha$ on $R[x_1,\ldots,x_n;\delta_1,\ldots,\delta_n]$, then $R[x_1,\ldots,x_n;\delta_1,\ldots,\delta_n]^{\widehat\alpha}=R^\alpha[x_1,\ldots,x_n;\delta_1,\ldots,\delta_n]$.
\end{proposition}

\begin{proof}
As additive groups, $R[x_1,\ldots,x_n;\delta_1,\ldots,\delta_n]^{\widehat\alpha}=R^\alpha[x_1,\ldots,x_n;\delta_1,\ldots,\delta_n]$, so we only need to show that the two also have the same multiplication. To this end, it suffices to show that $\big(rx_1^{i_1}\cdots x_n^{i_n}\big) * \big(sx_1^{j_1}\cdots x_n^{j_n}\big)=\big(rx_1^{i_1}\cdots x_n^{i_n}\big) \big(sx_1^{j_1}\cdots x_n^{j_n}\big)$ for any $r,s\in R$ and $i_1,\ldots,i_n, j_1,\ldots,j_n\in\N$. By using that $\alpha$ commutes with $\delta_1,\ldots,\delta_n$,
\begin{align*}
&\big(rx_1^{i_1}\cdots x_n^{i_n}\big) * \big(sx_1^{j_1}\cdots x_n^{j_n}\big)=\widehat\alpha\big(rx_1^{i_1}\cdots x_n^{i_n} sx_1^{j_1}\cdots x_n^{j_n}\big)\\
&=\alpha(r)x_1^{i_1}\cdots x_n^{i_n}\alpha(s)x_1^{j_1}\cdots x_n^{j_n}\\
&\stackrel{\eqref{eq:multi-multiplication}}{=}\sum_{\ell_1=0}^{i_1}\cdots\sum_{\ell_n=0}^{i_n}\binom{i_1}{\ell_1}\cdots\binom{i_n}{\ell_n}\alpha(r)\delta_1^{i_1-\ell_1}\circ \cdots \circ\delta_n^{i_n-\ell_n}(\alpha(s))x_1^{j_1+\ell_1}\cdots x_n^{j_n+\ell_n}\\
&=\sum_{\ell_1=0}^{i_1}\cdots\sum_{\ell_n=0}^{i_n}\binom{i_1}{\ell_1}\cdots\binom{i_n}{\ell_n}\alpha(r)\alpha\big(\delta_1^{i_1-\ell_1}\circ \cdots \circ\delta_n^{i_n-\ell_n}(s)\big)x_1^{j_1+\ell_1}\cdots x_n^{j_n+\ell_n}\\
&=\sum_{\ell_1=0}^{i_1}\cdots\sum_{\ell_n=0}^{i_n}\binom{i_1}{\ell_1}\cdots\binom{i_n}{\ell_n}\big(r*\delta_1^{i_1-\ell_1}\circ \cdots \circ\delta_n^{i_n-\ell_n}(s)\big)x_1^{j_1+\ell_1}\cdots x_n^{j_n+\ell_n}\\
&=\big(rx_1^{i_1}\cdots x_n^{i_n}\big) \big(sx_1^{j_1}\cdots x_n^{j_n}\big).\qedhere
\end{align*}
\end{proof}

\section{Families of higher-order hom-associative Weyl algebras}\label{sec:hom-Weyl}
In this section, we define, with the help of the previous section, hom-associative analogues of the higher-order Weyl algebras and investigate what basic properties they have. We begin by proving the following lemma:

\begin{lemma}\label{lem:unique-commuting-endomorphism}
Up to a constant $k\colonequals (k_1,\dots,k_n)\in K^n$, there is a unique $K$-algebra endomorphism $\alpha_{k}$ on $K[y_1,\ldots,y_n]$ that commutes with $\partial/\partial y_1,\ldots \partial/\partial y_n$, defined by $\alpha_k (y_\ell)=y_\ell+k_\ell$ for $1\leq \ell\leq n$.
\end{lemma}

\begin{proof}
Assume that $\alpha$ is a $K$-algebra endomorphism on $K[y_1,\ldots,y_n]$ that commutes with $\partial/\partial y_1,\ldots,\partial/\partial y_n$. Then $(\partial/\partial y_j)\alpha(y_\ell)=\alpha((\partial/\partial y_j)y_\ell)=\alpha(\delta_{j\ell})=\delta_{j\ell}$. Now, as a $K$-vector space, $K[y_1,\ldots,y_n]$ has a basis consisting of all monomials $y_1^{i_1}\cdots y_n^{i_n}$, $i_1,\ldots,i_n\in\N$. Hence, if we let $\alpha(y_\ell)=\sum_{i_1,\ldots,i_n\in\N} k'_{i_1,\ldots,i_n;\ell}y_1^{i_1}\cdots y_n^{i_n}$ for some $k'_{i_1,\ldots,i_n;\ell}\in K$ and set $0y_1^{i_1}\cdots y_j^{-1}\cdots y_n^{i_n}\colonequals 0$, then $(\partial/\partial y_j)\alpha(y_\ell)=\sum_{i_1,\ldots,i_n\in\N} i_jk'_{i_1,\ldots,i_n;\ell}y_1^{i_1}\cdots y_j^{i_j-1}\cdots y_n^{i_n}$. By comparing coefficients, we see that $(\partial/\partial y_j)\alpha(y_\ell)=\delta_{j\ell}$ if and only if $k'_{i_1,\ldots,i_n;\ell}$ is zero unless $i_1=\cdots =i_\ell-1=\cdots=i_n=0$, in which case $k'_{0,\ldots,1,\ldots,0;\ell}=1$, or $i_1=\cdots=i_n=0$. We conclude that $\alpha(y_\ell)=y_\ell+k_\ell$ for some $k_\ell\colonequals k'_{0,\ldots,0;\ell}\in K$, which defines $\alpha$ uniquely as a $K$-algebra endomorphism. To emphasize its dependence on $k\colonequals (k_1,\ldots,k_n)$, we denote $\alpha$ by $\alpha_k$.

Now assume that $\alpha_k$ is a $K$-algebra endomorphism defined by $\alpha_k(y_\ell)=y_\ell+k_\ell$ for some $k_\ell\in K$ and $1\leq \ell\leq n$. We claim that $\alpha_k$ commutes with $\partial/\partial y_1,\ldots \partial/\partial y_n$. To prove our claim, it suffices to show that $\alpha_k\big(y_1^{i_1}\cdots y_n^{i_n}\big)$ commutes with $\partial/\partial y_\ell$ for $i_1,\ldots,i_n\in\N$ and $1\leq \ell\leq n$. To this end,
\begin{align*}
 &(\partial/\partial y_\ell)\alpha_k(y_1^{i_1}\cdots y_n^{i_n})=(\partial/\partial y_\ell)\big(\alpha_k(y_1)^{i_1}\cdots\alpha_k(y_n)^{i_n}\big)\\&=(\partial/\partial y_\ell)\big((y_1+k_1)^{i_1}\cdots(y_n+k_n)^{i_n}\big)\\
 &=i_\ell(y_1+k_1)^{i_1}\cdots(y_\ell+k_\ell)^{i_\ell-1}\cdots(y_n+k_n)^{i_n}\\
&=i_\ell\alpha_k(y_1)^{i_1}\cdots\alpha_k(y_\ell)^{i_\ell-1}\cdots\alpha_k(y_n)^{i_n}\\&=\alpha_k\big(i_\ell y_1^{i_1}\cdots y_\ell^{i_\ell-1}\cdots y_n\big)=\alpha_k\big((\partial/\partial y_\ell)(y_1^{i_1}\cdots y_n^{i_n})\big).\qedhere
\end{align*}
\end{proof}

By using \autoref{lem:iterated-homogeneous-extension} and \autoref{lem:unique-commuting-endomorphism} together with \autoref{prop:yau-twisted-Ore-extension} we may now, in a unique way, define a hom-associative $n$-Weyl algebra as the hom-associative iterated differential polynomial ring $K[y_1\ldots,y_n][x_1,\ldots,x_n;\partial/\partial y_1,\ldots,\partial/\partial y_n]^{\widehat{\alpha}_k}=K[y_1\ldots,y_n]^{\alpha_k}[x_1,\ldots,x_n;\partial/\partial y_1,\ldots,\partial/\partial y_n]$; that is, as $A_n^{\widehat\alpha_k}$. To ease the notation, we write $A_n^k$ for $A_n^{\widehat\alpha_k}$. Moreover, for each $n\in\N_{>0}$, we may also view $k$ as a map $\{\ell\in\N_{>0}\mid \ell\leq n\}\to K$, $\ell\mapsto k_\ell$. In particular, $k$ depends on $n$, something we have deliberately omitted in our notation to avoid a cluttered ditto. With these notations, we thus have the following definition:

\begin{definition}[The hom-associative $n$-Weyl algebra]\label{def:hom-weyl}
The \emph{hom-associative $n$-Weyl algebra} is the hom-associative $K$-algebra $A_n^{k}$ where $k\in K^n$.
\end{definition}

Note that there are two families $(A_n^k)_{k\in K^n}$ and $(A_n^k)_{n\in\N_{>0}}$ of hom-associative Weyl algebras where $A_n^0=A_n$; here $0$ denotes the $n$-tuple $(0,\ldots,0)$. By the next result, the members of the latter family are naturally ordered by inclusion, so it is also natural to call $A_n^k$ the \emph{$n$th hom-associative Weyl algebra}.

\begin{remark}
The definition of the $n$th hom-associative Weyl algebra above coincides with that introduced in \cite{BR24}.
\end{remark}

\begin{lemma}
There is a chain $A_1^{k}\subsetneq\cdots\subsetneq A_n^{k}$ of hom-associative Weyl algebras.
\end{lemma}

\begin{proof}
Since we have a chain $A_1\subsetneq\cdots\subsetneq A_n$ of Weyl algebras and $\alpha_{(k_1,\ldots,k_\ell)}\vert A_{\ell-1} = \alpha_{(k_1,\ldots,k_{\ell-1})}$ for $2\leq \ell\leq n$, we see that the multiplication on $A_{\ell-1}^k$ is precisely the multiplication on $A_\ell^k$ restricted to the underlying $K$-vector space of $A_{\ell-1}$. Hence we have a chain $A_1^{k}\subsetneq\cdots\subsetneq A_n^{k}$ of hom-associative Weyl algebras.
\end{proof}

Since $k_1\frac{\partial}{\partial y_1},\ldots,k_n\frac{\partial}{\partial y_n}$ are all locally nilpotent $K$-linear maps, we may define $e^{k_1\frac{\partial}{\partial y_1}},\ldots,e^{k_n\frac{\partial}{\partial y_n}}$ as their formal power series (see \autoref{subsec:non-assoc}). If we introduce the symbol $k\frac{\partial}{\partial y
}\colonequals \sum_{\ell=1}^nk_\ell\frac{\partial}{\partial y_\ell}$, then by \autoref{prop:nilpotent-maps}, $e^{k_1\frac{\partial}{\partial y_1}}\cdots e^{k_n\frac{\partial}{\partial y_n}}=e^{k\frac{\partial}{\partial y}}$. If we also set $ik\colonequals (ik_1,\ldots,ik_n)$ for any $i\in\mathbb{Z}$, then we have the following lemma:

\begin{lemma}\label{lem:alpha-properties}
The following assertions hold:
\begin{enumerate}[label=\upshape(\roman*)]
\item $\widehat\alpha_k=e^{k\frac{\partial}{\partial y}}\in \Aut_K(A_n)$ with $\widehat\alpha_k^{-1}=e^{-k\frac{\partial}{\partial y}}$.\label{it:alpha1}
\item $\widehat\alpha_k^{i}=\widehat\alpha_{ik}$ for any $i\in\mathbb{Z}$.\label{it:alpha2}
\item $p*q=e^{k\frac{\partial}{\partial y}}(pq)$ for any $p,q \in A_n^k$.\label{it:alpha3}
\item $pq=e^{-k\frac{\partial}{\partial y}}(p*q)$  for any $p,q \in A_n^k$.\label{it:alpha4}
\end{enumerate}
\end{lemma}

\begin{proof}
\ref{it:alpha1}: To prove that $\widehat\alpha_k=e^{k\frac{\partial}{\partial y}}$, it suffices to show $\widehat\alpha_k\big(y_1^{i_1}\cdots y_n^{i_n}x_1^{j_1}\cdots x_n^{j_n}\big)=e^{k\frac{\partial}{\partial y}}\big(y_1^{i_1}\cdots y_n^{i_n}x_1^{j_1}\cdots x_n^{j_n}\big)$ for any $i_1,\ldots i_n, j_1,\ldots,j_n\in\N$. To this end, by using the binomial theorem and \autoref{prop:nilpotent-maps},
\begin{align*}
&\widehat\alpha_k\big(y_1^{i_1}\cdots y_n^{i_n}x_1^{j_1}\cdots x_n^{j_n}\big)=(y_1+k_1)^{i_1}\cdots (y_n+k_n)^{i_n}x_1^{j_1}\cdots x_n^{j_n}\\
&=\sum_{\ell_1=0}^{i_1}\binom{i_1}{\ell_1}k_1^{\ell_1}y_1^{i_1-\ell_1}\cdots \sum_{\ell_n=0}^{i_n}\binom{i_n}{\ell_n}k_n^{\ell_n}y_n^{i_n-\ell_n}x_1^{j_1}\cdots x_n^{j_n}\\
&=\sum_{\ell_1=0}^{i_1}\frac{(k_1(\partial/\partial y_1))^{\ell_1}}{\ell_1!}y_1^{i_1}\cdots\sum_{\ell_n=0}^{i_n}\frac{(k_n(\partial/\partial y_n))^{\ell_n}}{\ell_n!}y_n^{i_n}\big(x_1^{j_1}\cdots x_n^{j_n}\big)\\
&=\sum_{\ell_1=0}^{i_1}\frac{(k_1(\partial/\partial y_1))^{\ell_1}}{\ell_1!}\cdots\sum_{\ell_n=0}^{i_n}\frac{(k_n(\partial/\partial y_n))^{\ell_n}}{\ell_n!}\big(y_1^{i_1}\cdots y_n^{i_n}x_1^{j_1}\cdots x_n^{j_n}\big)\\
&= e^{k_1\frac{\partial}{\partial y_1}}\cdots e^{k_n\frac{\partial}{\partial y_n}}\big(y_1^{i_1}\cdots y_n^{i_n}x_1^{j_1}\cdots x_n^{j_n}\big)= e^{k\frac{\partial}{\partial y}}\big(y_1^{i_1}\cdots y_n^{i_n}x_1^{j_1}\cdots x_n^{j_n}\big).
\end{align*}
Hence $\widehat\alpha_k=e^{k\frac{\partial}{\partial y}}$. Now, we know that $\widehat\alpha_k$ is a $K$-algebra endomorphism on $A_n$. By the definition of $\widehat\alpha_k$, we have $\widehat\alpha_{-k}\circ\widehat\alpha_k=\widehat\alpha_k\circ\widehat\alpha_{-k}=\widehat\alpha_{k-k}=\widehat\alpha_0=\id_{A_n}$, so $\widehat\alpha_{-k}$ is both a left inverse and a right inverse to $\widehat\alpha_k$ which is therefore a bijection. We conclude that $\widehat\alpha_k\in \Aut_K(A_n)$.\\

\noindent\ref{it:alpha2}: This follows immediately from $\widehat\alpha_k=e^{k\frac{\partial}{\partial y}}$ and $\widehat\alpha_k^{-1}=e^{-k\frac{\partial}{\partial y}}$ in \ref{it:alpha1}.\\

\noindent\ref{it:alpha3}: By \ref{it:alpha1}, $\widehat\alpha_k=e^{k\frac{\partial}{\partial y}}$, so $p*q=e^{k\frac{\partial}{\partial y}}(pq)$ for any $p,q\in A_n^k$ by definition.\\

\noindent\ref{it:alpha4}: By \ref{it:alpha1}, $\widehat\alpha_k=e^{k\frac{\partial}{\partial y}}$ and $\widehat\alpha_k^{-1}=e^{-k\frac{\partial}{\partial y}}$. Hence the result follows from \ref{it:alpha3}.
\end{proof}

\begin{theorem}\label{thm:hom-Weyl-properties}
The following assertions hold:
    \begin{enumerate}[label=\upshape(\roman*)]
    \item $K$ embeds as a subfield into $A_n^k$.\label{it:subfield}
    \item $1_{A_n}$ is a unique weak identity element of $A_n^k$.\label{it:unique-weak-identity}
    \item $A_n^k$ contains no zero divisors.\label{it:no-zero-divisors}
    \item $A_n^k$ is power-associative if and only if $k=0$.\label{it:power-assoc}
    \end{enumerate}
\end{theorem}

\begin{proof}
\ref{it:subfield}: The proof is similar to that in \cite{BR20} for the hom-associative first Weyl algebra. However, we provide it here for the convenience of the reader. $K$ embeds into $A_n$ by the ring isomorphism $\varphi\colon K\to K'\colonequals \{k'y_1^0\cdots y_n^0 x_1^0\cdots x_n^0 \mid k'\in K\} \subseteq A_n$ defined by $\varphi(k')\colonequals k'y_1^0\cdots y_n^0x_1^0\cdots x_n^0$ for any $k'\in K$. The same map embeds $K$ into $A_n^k$, i.e. it is also an isomorphism of the hom-associative ring $K$ with twisting map $\id_K$, and the hom-associative subring $K' \subseteq A_n^k$ with twisting map $\widehat\alpha_k\vert K'=\id_{K'}$.\\

\noindent\ref{it:unique-weak-identity}: By \ref{it:alpha1} in \autoref{lem:alpha-properties}, $\widehat\alpha_k$ is injective, and so the result follows from \autoref{lem:unique-weak-identity}. \\

\noindent\ref{it:no-zero-divisors}: Since $\widehat\alpha_k$ is injective and $A_n$ contains no zero divisors, the result follows from \autoref{lem:zero-divisors}.\\

\noindent\ref{it:power-assoc}: If $k=0$, then $A_n^k$ is associative, and hence also power associative. Now assume that $k\neq0$. Then there is some $\ell$ where $1\leq \ell\leq n$ such that $k_\ell\neq0$, so
\begin{align*}
 &(y_\ell x_\ell*y_\ell x_\ell)*y_\ell x_\ell=((y_\ell+k_\ell)x_\ell(y_\ell+k_\ell)x_\ell)*y_\ell x_\ell\\
 &=((y_\ell+k_\ell)(y_\ell x_\ell+k_\ell x_\ell+1_{A_n})x_\ell)*y_\ell x_\ell\\
 &=(y_\ell+2k_\ell)((y_\ell+k_\ell)x_\ell+k_\ell x_\ell+1_{A_n})x_\ell(y_\ell+k_\ell)x_\ell\\
 &=(y_\ell+2k_\ell)((y_\ell+2k_\ell)x_\ell+1_{A_n})x_\ell(y_\ell+k_\ell)x_\ell\\
 &=(y_\ell+2k_\ell)((y_\ell+2k_\ell)x_\ell+1_{A_n})(y_\ell x_\ell+k_ \ell x_\ell+1_{A_n})x_\ell\\
  &=(y_\ell+2k_\ell)((y_\ell+2k_\ell)x_\ell+1_{A_n})((y_\ell+k_\ell)x_\ell+1_{A_n})x_\ell\\
  &=2k_\ell x_\ell+\text{[terms of higher total degree]},\\
&y_\ell x_\ell*(y_\ell x_\ell*y_\ell x_\ell)=y_\ell x_\ell*((y_\ell+k_\ell)(y_\ell x_\ell+k_\ell x_\ell+1_{A_n})x_\ell)\\
&=(y_\ell+k_\ell)x_\ell(y_\ell+2k_\ell)((y_\ell+k_\ell)x_\ell+k_\ell x_\ell+1_{A_n})x_\ell\\
&=(y_\ell+k_\ell)x_\ell(y_\ell+2k_\ell)((y_\ell+2k_\ell)x_\ell+1_{A_n})x_\ell\\
&=(y_\ell+k_\ell)(y_\ell x_\ell+2k_\ell x_\ell+1_{A_n})((y_\ell+2k_\ell)x_\ell+1_{A_n})x_\ell\\
&=(y_\ell+k_\ell)((y_\ell+2k_\ell)x_\ell+1_{A_n})((y_\ell+2k_\ell)x_\ell+1_{A_n})x_\ell\\
&=k_\ell x_\ell+\text{[terms of higher total degree]}.
\end{align*}
By comparing coefficients of the two terms with lowest total degree, we see that $(y_\ell x_\ell*y_k\ell x_\ell)*y_\ell x_\ell=y_\ell x_\ell*(y_\ell x_\ell*y_\ell x_\ell)$ only if $k_\ell=0$.
\end{proof}

\begin{remark}
From \ref{it:power-assoc} in \autoref{thm:hom-Weyl-properties} we can also conclude that $A_n^k$ is left alternative, right alternative, flexible, and associative (see \autoref{subsec:non-assoc}) if and only if $k=0$.
\end{remark}

In \cite[Corollary 3.7]{BR24}, the authors showed that $A_n^k$ is simple whenever $k_1\cdots k_n\neq 0$. The next theorem generalizes this result for any $k\in K^n$.

\begin{theorem}\label{thm:simple}
$A_n^k$ is simple.
\end{theorem}

\begin{proof}
Let $I$ be an ideal of $A_n^k$. We claim that $I$ is an ideal of $A_n$, and since $A_n$ is simple, $I$ must be equal to either the zero ideal of $A_n$ or $A_n$ itself. As $K$-vector spaces, the zero ideal of $A_n$ is equal to the zero ideal of $A_n^k$, and $A_n$ is equal to $A_n^k$, so if we can prove the claim, we are done. To prove the claim, assume further that $I$ is nonzero and pick a nonzero $p\in I$. We may define the \emph{degrees} $\deg_{y_i}$ and $\deg_{x_j}$ of $p$ in the indeterminates $y_i$ and $x_j$ where $1\leq i,j\leq n$, respectively, as follows. We let $\deg_{y_i}(p)$ ($\deg_{x_j}(p)$) be the largest exponent of $y_i$ (of $x_j$) that appears in the expression of $p$, when written in the standard basis of $A_n^k$, viewed as a $K$-vector space (see \autoref{subsec:iterated}).\\

\noindent Step 1. Pick an $i$, $1\leq i\leq n$, with $\deg_{y_i}(p)>0$. If there is no such $i$, meaning $p$ is a nonzero polynomial in the indeterminates $x_1,\ldots,x_n$, go to Step 2. Now, if we denote by $[\cdot,\cdot]_*$ the commutator in $A_n^k$, then $[x_i,p]_*=\widehat\alpha_k\left([x_i,p]\right)=[\widehat\alpha_k(x_i),\widehat\alpha_k(p)]=[x_i,\widehat\alpha_k(p)]$. From \autoref{subsec:iterated}, $[x_i,\widehat\alpha_k(p)]=\ad_{x_i}(\widehat\alpha_k(p))=(\partial/\partial y_i)\widehat\alpha_k(p)$. Now, by the definition of $\widehat\alpha_k$, we have $\deg_{y_i}(p)=\deg_{y_i}(\widehat\alpha_k(p))$. We conclude that $[x_i,p]_*$, which is a polynomial in $I$, has degree in $y_i$ equal to $\deg_{y_i}(p)-1$. By repeating this process, we eventually get a nonzero element of $I$ which has degree 0 in $y_i$. Now, rename this element $p$ and start over from Step 1.\\

\noindent Step 2. We have a nonzero polynomial $p$ in the indeterminates $x_1,\ldots,x_n$ which is an element of $I$. Pick a $j$, $1\leq j\leq n$ with $\deg_{x_j}(p)>0$. If there is no such $j$, meaning $p$ is a nonzero element of $K$, go to Step 3. Now, $[p,y_j]_*=\widehat\alpha_k([p,y_j])=[\widehat\alpha_k(p),\alpha_k(y_j)]=[\widehat\alpha_k(p),y_j+k_j]=[\widehat\alpha_k(p),y_j]$. From \autoref{subsec:iterated}, $[\widehat\alpha_k(p),y_j]=-\ad_{y_j}(\widehat\alpha_k(p))=(\partial/\partial x_j)\widehat\alpha_k(p)$. Moreover, $\deg_{x_j}(p)=\deg_{x_j}(\alpha_k(p))$. We conclude that $[p,y_j]_*$, which is a nonzero element of $I$, has degree in $x_j$ equal to $\deg_{x_j}(p)-1$. By repeating this process, we eventually get a nonzero element of $I$ which has degree 0 in $x_j$. Now, rename this element $p$ and start over from Step 2.\\

\noindent Step 3. We have a nonzero element $p$ of $K\cap I$. Then $1_{A_n}=p^{-1}*p\in I$, and since $\widehat\alpha_k$ is surjective, any $q\in A_n^k$ may be written as $\widehat\alpha_k(q')$ for some $q'\in A_n^k$. Then $q=\widehat\alpha_k(q')=1_{A_n}*q'\in I$, so $A_n^k\subseteq I$ and $I\subseteq A_n^k$. We conclude that $I=A_n^k$.
\end{proof}

\begin{theorem}\label{thm:hom-Weyl-properties2}
 The following assertions hold:
    \begin{enumerate}[label=\upshape(\roman*)]
    \item $C(A_n^k)=K$.\label{it:commuter}
    \item $N(A_n^k)=N_l(A_n^k)=N_m(A_n^k)=N_r(A_n^k)=\begin{cases}A_n^k&\text{if } k=0, \\\{0\}&\text{otherwise}.\end{cases}$\label{it:nuclei}
    \item $Z(A_n^k)=\begin{cases}K&\text{if } k=0, \\\{0\}&\text{otherwise}.\end{cases}$\label{it:center}
    \end{enumerate}
\end{theorem}

\begin{proof}
\ref{it:commuter}: Since $\widehat\alpha_k$ is injective and $C(A_n)=K$, the result follows from \autoref{lem:commuter}. \\

\noindent\ref{it:nuclei} Let $p\in N_l(A_n^k)$. Then $(p*1_{A_n})*1_{A_n}=\widehat\alpha_k(p)*1_{A_n}=\widehat\alpha_k(\widehat\alpha_k(p))$ while $p*(1_{A_n}*1_{A_n})=p*1_{A_n}=\widehat\alpha_k(p)$. Hence $(p*1_{A_n})*1_{A_n}=p*(1_{A_n}*1_{A_n})$ if and only if $\widehat\alpha_k(p)=\widehat\alpha_k(\widehat\alpha_k(p))$. By applying $\widehat\alpha_k^{-1}$ to both sides, the equality is equivalent to $p=\widehat\alpha_k(p)$. Now, for any $\ell$ with $1\leq \ell\leq n$, $(p*1_{A_n})*y_\ell=\widehat\alpha_k(p)*y_\ell=p*y_\ell=\widehat\alpha_k(p)(y_\ell+k_\ell)=p(y_\ell+k_\ell)$ while $p*(1_{A_n}*y_\ell)=p*(y_\ell+k_\ell)=\widehat\alpha_k(p)(y_\ell+2k_\ell)=p(y_\ell+2k_\ell)$. Hence $(p*1_{A_n})*y_\ell=p*(1_{A_n}*y_\ell)$ if and only if $k_\ell p=2k_\ell p$; that is, $k_\ell p=0$. If $k_\ell=0$ for $1\leq \ell\leq n$, then $k=0$, and so $A_n^k$ is associative with $N_l(A_n^k)=A_n^k$. If $p=0$, then $N_l(A_n^k)\subseteq \{0\}$, and since $\{0\}\subseteq N_l(A_n^k)$, we have $N_l(A_n^k)=\{0\}$.

Now assume that $p\in N_m(A_n^k)$. For any $\ell$ with $1\leq \ell \leq n$, $(y_\ell*p)*1_{A_n}=((y_\ell+k_\ell)\widehat\alpha_k(p))*1_{A_n}=(y_\ell+2k_\ell)\widehat\alpha_k(\widehat\alpha_k(p))=(y_\ell+2k_\ell)\widehat\alpha_{2k}(p)$ while $y_\ell*(p*1_{A_n})=y_\ell*\widehat\alpha_k(p)=(y_\ell+k_\ell)\widehat\alpha_{2k}(p)$. Hence $(y_\ell*p)*1_{A_n}=y_\ell*(p*1_{A_n})$ if and only if $2k_\ell\widehat\alpha_{2k}(p)=k_\ell\widehat\alpha_{2k}(p)$, which holds if and only if $k_\ell=0$ or $2\widehat\alpha_{2k}(p)=\widehat\alpha_{2k}(p)$; that is, $\widehat\alpha_{2k}(p)=0$. By \ref{it:alpha1} in \autoref{lem:alpha-properties}, $\widehat\alpha_{2k}$ is injective, so the latter equality holds if and only if $p=0$. If $k_\ell=0$ for $1\leq \ell\leq n$, then $N_m(A_n^k)=A_n^k$. If $p=0$, then $N_m(A_n^k)\subseteq \{0\}$, and since $\{0\}\subseteq N_m(A_n^k)$, we have $N_m(A_n^k)=\{0\}$.

Last, assume that $p\in N_r(A_n^k)$. Then $(1_{A_n}*1_{A_n})*p=\widehat\alpha_k(1_{A_n})*p=1_{A_n}*p=\widehat\alpha_k(p)$ while $1_{A_n}*(1_{A_n}*p)=1_{A_n}*\widehat\alpha_k(p)=\widehat\alpha_k(\widehat\alpha_k(p))$. Hence $(1_{A_n}*1_{A_n})*p= 1_{A_n}*(1_{A_n}*p)$ if and only if $\widehat\alpha_k(p)=\widehat\alpha_k(\widehat\alpha_k(p))$. By applying $\widehat\alpha^{-1}_k$ to both sides, the equality is equivalent to $p=\widehat\alpha_k(p)$. Now, for any $\ell$ with $1\leq \ell\leq n$, $(y_\ell*1_{A_n})*p=(y_\ell+k_\ell)*p=(y_\ell+2k_\ell)\widehat\alpha_k(p)=(y_\ell+2k_\ell)p$ while $y_\ell*(1_{A_n}*p)=y_\ell*\widehat\alpha_k(p)=y_\ell*p=(y_\ell+k_\ell)\widehat\alpha_\ell(p)=(y_\ell+k_\ell)p$. Hence $(y_\ell*1_{A_n})*p=y_\ell*(1_{A_n}*p)$ if and only if $2k_\ell p=k_\ell p$; that is, $k_\ell p=0$. If $k_\ell=0$ for $1\leq \ell\leq n$, then $k=0$, and so $A_n^k$ is associative with $N_r(A_n^k)=A_n^k$. If $p=0$, then $N_r(A_n^k)\subseteq \{0\}$, and since $\{0\}\subseteq N_r(A_n^k)$, we have $N_r(A_n^k)=\{0\}$.

Since $N(A_n^k)=N_l(A_n^k)\cap N_m(A_n^k)\cap N_r(A_n^k)$, the result now follows.
\\

\noindent\ref{it:center}: Since $Z(A_n^k)\colonequals C(A_n^k)\cap N(A_n^k)$, the result follows from \ref{it:commuter} and \ref{it:nuclei}.
\end{proof}

\begin{lemma}\label{lem:hom-Weyl-derivations}
The following assertions hold:
\begin{enumerate}[label=\upshape(\roman*)]
	\item $\Der_K(A_n^k)=\{\delta\in\Der_K(A_n)\mid  \delta\circ\widehat\alpha_k=\widehat\alpha_k\circ\delta\}$.\label{it:hom-Weyl-derivation1}
	\item For any $\beta\in\End_K(A_n)$, $\{\delta\in\Der_K(A_n)\mid  \delta\circ\beta=\beta\circ\delta\}=\{\delta\in\Der_K(A_n)\mid \delta(\beta(x_\ell))=\beta(\delta(x_\ell)),\ \delta(\beta(y_\ell))=\beta(\delta(y_\ell)),\ 1\leq \ell\leq n\}$.\label{it:hom-Weyl-derivation2}
\end{enumerate}
\end{lemma}

\begin{proof}
\ref{it:hom-Weyl-derivation1}: By \autoref{lem:derivations}, it suffices to show that $\delta(1_{A_n})=0$ for any $\delta\in\Der_K(A_n^k)$. To this end, let $\delta\in\Der_K(A_n^k)$ be arbitrary. Then $\delta(1_{A_n}*1_{A_n})=\delta(1_{A_n})*1_{A_n}+1_{A_n}*\delta(1_{A_n})=2\widehat\alpha_k(\delta(1_{A_n}))=2e^{k\frac{\partial}{\partial y}}\delta(1_{A_n})$. On the other hand, $\delta(1_{A_n}*1_{A_n})=\delta(\widehat\alpha_k(1_{A_n}))=\delta(1_{A_n})$. If we let $p\colonequals \delta(1_{A_n})$, then the equality of the two above expressions is equivalent to the eigenvector problem $e^{k\frac{\partial}{\partial y}}p=\frac{1}{2}p$. We claim that it has no solution, that is,  $p=0$. This can be seen as follows. The above equation is equivalent to the PDE $p=-2\sum_{i\in\N_{>0}} \frac{1}{i!}\big(k\frac{\partial}{\partial y}\big)^ip$,
where $k\frac{\partial}{\partial y
}\colonequals \sum_{j=1}^nk_j\frac{\partial}{\partial y_j}$. Since $k_1\frac{\partial}{\partial y_1},\ldots,k_n\frac{\partial}{\partial y_n}$ are pairwise commuting locally nilpotent $K$-linear maps, from the proof of \autoref{prop:nilpotent-maps}, $\big(k\frac{\partial}{\partial y}\big)^i$ is also locally nilpotent. Hence the sum in the above PDE is in fact finite. If $p\neq0$, then the left-hand side of the PDE has higher total degree than the corresponding right-hand side, which is a contradiction. We conclude that $p=0$.\\

\noindent\ref{it:hom-Weyl-derivation2} Let $\beta\in\End_K(A_n)$. Then $\{\delta\in\Der_K(A_n)\mid \delta\circ\beta = \beta\circ\delta\}\subseteq\{\delta\in\Der_K(A_n)\mid \delta(\beta(x_\ell))=\beta(\delta(x_\ell)),\ \delta(\beta(y_\ell))=\beta(\delta(y_\ell)),\ 1\leq \ell\leq n\}$, so we only need to show that the opposite inclusion also holds. To this end, let $\delta\in \Der_K(A_n)$ and assume that $\delta(\beta(x_\ell))=\beta(\delta(x_\ell))$ and $\delta(\beta(y_\ell))=\beta(\delta(y_\ell))$ for $1\leq \ell\leq n$. By the $K$-linearity of $\beta$ and $\delta$, to show that $\delta\circ\beta = \beta\circ\delta$, it suffices to show that the equality $\delta\big(\beta\big(y_1^{i_1}\cdots y_n^{i_n}x_1^{j_1}\cdots x_n^{j_n}\big)\big)=\beta\big(\delta\big(y_1^{i_1}\cdots y_n^{i_n}x_1^{j_1}\cdots x_n^{j_n}\big)\big)$ holds for any $i_1,\ldots,i_n,$ $j_1,\ldots,j_n\in\N$. To show that the equality holds, we use induction on the total degree $I\colonequals i_1+\cdots+i_n\in\N$ of the indeterminates $y_1,\ldots,y_n$. To show the base case $I=0$, we again use induction, this time on the total degree $J\colonequals j_1+\cdots + j_n\in\N$ of the indeterminates $x_1,\ldots,x_n$. The base case $J=0$ is immediate, since $\delta(\beta(1_{A_n})=\delta(1_{A_n})=0=\beta(0)=\beta(\delta(1_{A_n}))$.
Now, note that the induction case $J+1$ implies that there is some $\ell$ for which $x_\ell$ has degree $i_\ell+1$. Since $x_1^{j_1},\ldots,x_n^{j_n}$ all commute, 
\begin{align*}
&\delta\big(\beta\big(x_1^{j_1}\cdots x_\ell^{j_{\ell+1}}\cdots x_n^{j_n}\big)\big)=\delta\big(\beta\big(x_1^{j_1}\cdots x_n^{j_n}x_\ell\big)\big)=\delta\big(\beta\big(x_1^{j_1}\cdots x_n^{j_n}\big)\beta(x_\ell)\big)\\
&=\beta\big(x_1^{j_1}\cdots x_n^{j_n}\big)\delta(\beta(x_\ell))+\delta\big(\beta\big(x_1^{j_1}\cdots x_n^{j_n}\big)\beta(x_\ell)\\
&=\beta\big(x_1^{j_1}\cdots x_n^{j_n}\big)\beta(\delta(x_\ell))+\beta\big(\delta\big(x_1^{j_1}\cdots x_n^{j_n}\big)\beta(x_\ell)\\
&=\beta\big(x_1^{j_1}\cdots x_n^{j_n}\delta(x_\ell)+\delta\big(x_1^{j_1}\cdots x_n^{j_n}\big)x_\ell\big)\\
&=\beta\big(\delta\big(x_1^{j_1}\cdots x_n^{j_n}x_\ell\big)\big)=\beta\big(\delta\big(x_1^{j_1}\cdots x_\ell^{j_{\ell+1}}\cdots x_n^{j_n}\big)\big).
\end{align*}
Hence we have shown the induction step $J+1$, and so the base case $I=0$ holds. Since $y_1^{i_1},\ldots,y_n^{i_n}$ all commute,

\begin{align*}
&\delta\big(\beta\big(y_1^{i_1}\cdots y_\ell^{\ell+1}\cdots y_n^{i_n} x_1^{j_1}\cdots x_n^{j_n}\big)\big)=\delta\big(\beta\big(y_1^{i_1}\cdots y_n^{i_n}y_\ell x_1^{j_1}\cdots x_n^{j_n}\big)\big)\\
&=\delta\big(\beta\big(y_1^{i_1}\cdots y_n^{i_n}\big)\beta\big(y_\ell x_1^{j_1}\cdots x_n^{j_n}\big)\big)\\
&=\beta\big(y_1^{i_1}\cdots y_n^{i_n}\big)\delta\big(\beta\big(y_\ell x_1^{j_1}\cdots x_n^{j_n}\big)\big)+\delta\big(\beta\big(y_1^{i_1}\cdots y_n^{i_n}\big)\big)\beta\big(y_\ell x_1^{j_1}\cdots x_n^{j_n}\big)\\
&=\beta\big(y_1^{i_1}\cdots y_n^{i_n}\big)\delta\big(\beta(y_\ell)\beta\big(x_1^{j_1}\cdots x_n^{j_n}\big)\big)+\delta\big(\beta\big(y_1^{i_1}\cdots y_n^{i_n}\big)\big)\beta\big(y_\ell x_1^{j_1}\cdots x_n^{j_n}\big)\\
&\phantom{=\ }+\delta\big(\beta\big(y_1^{i_1}\cdots y_n^{i_n}\big)\big)\beta\big(y_\ell x_1^{j_1}\cdots x_n^{j_n}\big)\\
&=\beta\big(y_1^{i_1}\cdots y_n^{i_n}\big)\big(\beta(y_\ell)\delta\big(\beta\big(x_1^{j_1}\cdots x_n^{j_n}\big)+\delta(\beta(y_\ell))\beta\big(x_1^{j_1}\cdots x_n^{j_n}\big)\big)\\
&\phantom{=\ }+\beta\big(\delta\big(y_1^{i_1}\cdots y_n^{i_n}\big)\big)\beta\big(y_\ell x_1^{j_1}\cdots x_n^{j_n}\big)	\\
&=\beta\big(y_1^{i_1}\cdots y_n^{i_n}\big)\big(\beta(y_\ell) \beta\big(\delta\big(x_1^{j_1}\cdots x_n^{j_n}\big)+\beta(\delta(y_\ell))\beta\big(x_1^{j_1}\cdots x_n^{j_n}\big)\big)\\
&=\beta\big(y_1^{i_1}\cdots y_n^{i_n}\big(y_\ell \delta\big(x_1^{j_1}\cdots x_n^{j_n}\big)+\delta(y_\ell)x_1^{j_1}\cdots x_n^{j_n}\big)+\delta\big(y_1^{i_1}\cdots y_n^{i_n}\big)y_\ell x_1^{j_1}\cdots x_n^{j_n}\big)\\
&=\beta\big(y_1^{i_1}\cdots y_n^{i_n}\delta\big(y_\ell x_1^{j_1}\cdots x_n^{j_n}\big)+\delta\big(y_1^{i_1}\cdots y_n^{i_n}\big)y_\ell x_1^{j_1}\cdots x_n^{j_n}\big)\\
&=\beta\big(\delta\big(y_1^{i_1}\cdots y_n^{i_n}y_\ell x_1^{j_1}\cdots x_n^{j_n}\big)\big)=\beta\big(\delta\big(y_1^{i_1}\cdots y_\ell^{i_\ell+1}\cdots y_n^{i_n}x_1^{j_1}\cdots x_n^{j_n}\big)\big).
\end{align*}
Hence we have shown the induction step $I+1$, which concludes the proof.
\end{proof}

We denote by $k^{-1}(0)$ the \emph{zero set} of $k$, seen as a map $\{\ell\in\N_{>0}\mid \ell\leq n\}\to K$, $\ell\mapsto k_\ell$, that is, the set $\{\ell\in\N_{>0}\mid k_\ell=0,\ \ell\leq n\}$. We then write $y_{k^{-1}(0)}$ for the $\vert k^{-1}(0)\vert$-tuple $(y_i)_{i\in k^{-1}(0)}$, but without parentheses. For example, if $k=(0,0,k_3,\ldots,k_n)$ with $k_3\cdots k_n\neq0$, then $y_{k^{-1}(0)}=y_1,y_2$. With this notation, we have the following result:

\begin{theorem}\label{thm:derivations}$\Der_K(A_n^k)=\big\{\ad_p\in\InnDer_K(A_n)\mid p=\sum_{i\not\in k^{-1}(0)}f_iy_i \\+ q(y_{k^{-1}(0)},x_1,\ldots,x_n)\in A_n,\ f_i\in K\big\}$.
\end{theorem}

\begin{proof}
By \ref{it:hom-Weyl-derivation1} in \autoref{lem:hom-Weyl-derivations}, $\Der_K(A_n^k)=\{\delta\in\Der_K(A_n)\mid \delta\circ\widehat\alpha_k=\widehat\alpha_k\circ\delta\}$, and by \ref{it:hom-Weyl-derivation2} in \autoref{lem:hom-Weyl-derivations}, $\{\delta\in\Der_K(A_n)\mid  \delta\circ\widehat\alpha_k=\widehat\alpha_k\circ\delta\}=\{\delta\in\Der_K(A_n)\mid \delta(\widehat\alpha_k(x_\ell))=\widehat\alpha_k(\delta(x_\ell)),\ \delta(\alpha_k(y_\ell))=\widehat\alpha_k(\delta(y_\ell)),\ 1\leq \ell\leq n\}$. Now let $\delta\in\Der_K(A_n^k)$. From \autoref{subsec:iterated} we have $\Der_K(A_n)=\InnDer_K(A_n)$, so $\delta=\ad_p$ for some $p\in A_n$ that satisfies $\widehat\alpha_k(\ad_p(y_\ell))=\ad_p(\alpha_k(y_\ell))=\ad_p(y_\ell+k_\ell)=\ad_p(y_\ell)$ and $\widehat\alpha_k(\ad_p(x_\ell))=\ad_p(\widehat\alpha_k(x_\ell))=\ad_p(x_\ell)$. If we use $\widehat\alpha_k=e^{k\frac{\partial}{\partial y}}$ from \ref{it:alpha1} in \autoref{lem:alpha-properties}, then we have $2n$ eigenvector problems $e^{k\frac{\partial}{\partial y}}s=s$, $s\in \{\ad_p(y_\ell)\mid 1\leq \ell\leq n\}\cup \{\ad_p (x_\ell)\mid 1\leq \ell\leq n\}$. The above equation is equivalent to the PDE $\sum_{i\in\N_{>0}} \frac{1}{i!}\big(k\frac{\partial}{\partial y}\big)^is=0$ where $k\frac{\partial}{\partial y
}\colonequals \sum_{j=1}^nk_j\frac{\partial}{\partial y_j}$. From \autoref{subsec:iterated}, $\ad_p(y_\ell)=-\ad_{y_\ell}(p)=\frac{\partial}{\partial x_\ell}p$ and $\ad_p(x_\ell)=-\ad_{x_\ell}(p)=-\frac{\partial}{\partial y_\ell}p$. Hence we have $2n$ PDEs $\sum_{i\in\N_{>0}} \frac{1}{i!}\big(k\frac{\partial}{\partial y}\big)^i\frac{\partial}{\partial y_\ell}p=0$ and $\sum_{i\in\N_{>0}} \frac{1}{i!}\big(k\frac{\partial}{\partial y}\big)^i\frac{\partial}{\partial x_\ell}p=0$, $1\leq \ell\leq n$, which in turn are equivalent to the following PDEs: $\frac{\partial }{\partial y_\ell}\sum_{i\in\N_{>0}} \frac{1}{i!}\big(k\frac{\partial}{\partial y}\big)^ip=0$ and $\frac{\partial }{\partial x_\ell}\sum_{i\in\N_{>0}} \frac{1}{i!}\big(k\frac{\partial}{\partial y}\big)^ip=0$, $1\leq \ell\leq n$. From the first PDEs, $\sum_{i\in\N_{>0}} \frac{k^i}{i!}\frac{\partial^i}{\partial y^i}p\in K[x_1,\ldots,x_n]$, and so from the latter ones, $\sum_{i\in\N_{>0}} \frac{1}{i!}\big(k\frac{\partial}{\partial y}\big)^ip\in K$. By comparing coefficients of the monomials with highest degree, we see that $p=\sum_{i\not\in k^{-1}(0)}f_iy_i+q(y_{k^{-1}(0)},x_1,\ldots,x_n)\in A_n$, $f_i\in K$.
\end{proof}

\begin{lemma}\label{lem:hom-Weyl-homomorphisms}
The following assertions hold:
\begin{enumerate}[label=\upshape(\roman*)]
	\item $\Hom_K(A_n^k, A_n^{k'})=\{\varphi\in\End_K(A_n)\mid \varphi\circ \widehat\alpha_k=\widehat\alpha_{k'}\circ \varphi\}$.\label{it:hom-Weyl-homomorphisms1}
	\item For any $\beta,\beta'\in\End_K(A_n)$, $\{\varphi\in\End_K(A_n)\mid   \varphi\circ \beta=\beta'\circ \varphi\}=\{\varphi\in\End_K(A_n)\mid \varphi(\beta(x_\ell))=\beta'(\varphi(x_\ell)),\ \varphi(\beta(y_\ell))=\beta'(\varphi(y_\ell)),\ 1\leq \ell\leq n\}$.\label{it:hom-Weyl-homomorphisms2}
\end{enumerate}
\end{lemma}

\begin{proof}
\ref{it:hom-Weyl-homomorphisms1}: Since $\widehat\alpha_{k'}$ is injective, this follows immediately from \autoref{lem:morphisms}.\\

\noindent\ref{it:hom-Weyl-homomorphisms2}: Let $\beta, \beta'\in\End_K(A_n)$. Then $\{\varphi\in\End_K(A_n)\mid \varphi\circ \beta=\beta'\circ \varphi\}
\subseteq\{\varphi\in\End_K(A_n)\mid\varphi(\beta(x_\ell))=\beta'(\varphi(x_\ell)),\ \varphi(\beta(y_\ell))=\beta'(\varphi(y_\ell)),\ 1\leq \ell\leq n\}$, so we only need to show that the opposite inclusion also holds. To this end, let $\varphi\in \End_K(A_n)$ and assume that $\varphi(\beta(x_\ell))=\beta'(\varphi(x_\ell))$ and $\varphi(\beta(y_\ell))=\beta'(\varphi(y_\ell))$ for $1\leq \ell\leq n$. By the $K$-linearity of $\beta, \beta'$, and $\varphi$, to show that $\varphi\circ \beta
=\beta'\circ \varphi$, it suffices to show that the equality $\varphi\big(\beta\big(y_1^{i_1}\cdots y_n^{i_n}x_1^{j_1}\cdots x_n^{j_n}\big)\big)=\beta'\big(\varphi\big(y_1^{i_1}\cdots y_n^{i_n}x_1^{j_1}\cdots x_n^{j_n}\big)\big)$ holds for any $i_1,\ldots,i_n,$ $j_1,\ldots,j_n\in\N$. Since $\beta,\beta'$, and $\varphi$ are $K$-algebra endomorphisms, 
\begin{align*}
&\varphi\big(\beta\big(y_1^{i_1}\cdots y_n^{i_n}x_1^{j_1}\cdots x_n^{j_n}\big)\big)=\varphi(\beta(y_1))^{i_1}\cdots \varphi(\beta(y_n))^{i_n}\varphi(\beta(x_1))^{j_1}\cdots \varphi(\beta(x_n))^{j_n}\\
&\beta'(\varphi(y_1))^{i_1}\cdots \beta'(\varphi(y_n))^{i_n}\beta'(\varphi(x_1))^{j_1}\cdots \beta'(\varphi(x_n))^{j_n}=\beta'\big(\varphi\big(y_1^{i_1}\cdots y_n^{i_n}x_1^{j_1}\cdots x_n^{j_n}\big)\big).	
\end{align*}
\end{proof}

\begin{proposition}\label{prop:Weyl-homomorphisms}Any $\varphi\in\Hom_K(A_n^k, A_n^{k'})$ is defined by 
\[
\varphi(x_\ell)=p_\ell+\sum_{i=1}^nf_{i\ell}y_i, \quad\qquad\varphi(y_\ell)=q_\ell+\sum_{i=1}^ng_{i\ell}y_i,
\]
 for some $p_\ell\colonequals p_\ell(y_{{k'}^{-1}(0)},x_1,\ldots,x_n),\ q_\ell\colonequals q_\ell(y_{{k'}^{-1}(0)},x_1,\ldots,x_n)\in A_n$, and $f_{i\ell}, g_{i\ell}\in K$, satisfying, for $1\leq j,\ell\leq n$, the following equations:
\begin{align}
&\sum_{i=1}^nf_{i\ell}k_i=0,\label{eq:coeff1}\\
&\sum_{i=1}^ng_{i\ell}k'_i=k_\ell,\label{eq:coeff2}\\
&\sum_{i=1}^n\frac{\partial}{\partial x_i}(f_{i\ell}p_j-f_{ij}p_\ell)=[p_\ell,p_j],\label{eq:PDE1}\\
&\sum_{i=1}^n\frac{\partial}{\partial x_i}(g_{i\ell}q_j-g_{ij}q_\ell)=[q_\ell,q_j],\label{eq:PDE2}\\
&\sum_{i=1}^n\frac{\partial}{\partial x_i}(g_{i\ell}p_j-f_{ij}q_\ell)=[q_\ell,p_j]+\delta_{j\ell}.\label{eq:PDE3}
\end{align}
\end{proposition}

\begin{proof}
By \ref{it:hom-Weyl-homomorphisms1} in \autoref{lem:hom-Weyl-homomorphisms}, $\Hom_K(A_n^k, A_n^{k'})=\{\varphi\in\End_K(A_n)\mid \varphi\circ \widehat\alpha_k = \widehat\alpha_{k'}\circ \varphi \}$, and by \ref{it:hom-Weyl-homomorphisms2} in \autoref{lem:hom-Weyl-homomorphisms}, $\{\varphi\in\End_K(A_n)\mid \varphi\circ \widehat\alpha_k = \widehat\alpha_{k'}\circ \varphi \}=\{\varphi\in\End_K(A_n)\mid \varphi(\widehat\alpha_k(x_\ell))=\widehat\alpha_{k'}(\varphi(x_\ell)),\ \varphi(\alpha_k(y_\ell))=\widehat\alpha_{k'}(\varphi(y_\ell)),\ 1\leq \ell\leq n\}$. Now let $\varphi\in\Hom_K(A_n^k,A_n^{k'})$. Then $\varphi\in\End_K(A_n)$ with $\varphi(x_\ell)=\varphi(\widehat\alpha_k(x_\ell))=\widehat\alpha_{k'}(\varphi(x_\ell))=e^{k'\frac{\partial}{\partial y}}\varphi(x_\ell)$ and $\varphi(y_\ell)+k_\ell=\varphi(y_\ell+k_\ell)=\varphi(\alpha_k(y_\ell))=\widehat\alpha_{k'}(\varphi(y_\ell))=e^{k'\frac{\partial}{\partial y}}\varphi(y_\ell)$ for $1\leq \ell\leq n$. The above equations are equivalent to the following $2n$ PDEs: $\sum_{i\in\N_{0}} \frac{1}{i!}\big(k'\frac{\partial}{\partial y}\big)^i\varphi(x_\ell)=0$ and $\sum_{i\in\N_{>0}} \frac{1}{i!}\big(k'\frac{\partial}{\partial y}\big)^i\varphi(y_\ell)=k_\ell$ for $1\leq \ell\leq n$, where $k'\frac{\partial}{\partial y
}\colonequals \sum_{j=1}^nk'_j\frac{\partial}{\partial y_j}$. From the first PDEs, $\varphi(x_\ell)=p_\ell(y_{k'^{-1}(0)},x_1,\ldots,x_n)+\sum_{i=1}^nf_{i\ell}y_i$ for some $f_{i\ell}\in K$ where $\sum_{i=1}^nf_{i\ell}k_i=0$ and some $p_\ell(y_{k'^{-1}(0)},x_1,\ldots,x_n)\in A_n$. From the second PDEs, $\varphi(y_\ell)=q_\ell(y_{k'^{-1}(0)},x_1,\ldots,x_n)+\sum_{i=1}^ng_{i\ell}y_i$ for some $g_{i\ell}\in K$ where $\sum_{i=1}^ng_{i\ell}k'_i =k_\ell$ and some $q_\ell(y_{k'^{-1}(0)},x_1,\ldots,x_n)\in A_n$. Recall that $A_n$ is the free, associative, and unital $K$-algebra on the letters $x_1,\ldots,x_n,y_1,\ldots,y_n$ modulo the commutation relations $[x_j,x_\ell]=[y_j,y_\ell]=0$ and $[x_j,x_\ell]=\delta_{j\ell}$ for $1\leq j,\ell\leq n$. Hence, $\varphi\in\End_K(A_n)$ if and only if $[\varphi(x_j),\varphi(x_\ell)]=\varphi([x_j,x_\ell])=\varphi([y_j,y_\ell])=[\varphi(y_j),\varphi(y_\ell)]=\varphi(0)=0=$ and $[\varphi(x_j),\varphi(y_\ell)]=\varphi([x_j,y_\ell])=\varphi(\delta_{j\ell})=\delta_{j\ell}$ for $1\leq j,\ell\leq n$. Since the commutator is linear in both arguments and $\ad_{y_i}=-\partial/\partial x_i$ for $1\leq i\leq n$, the above relations are equivalent to the PDEs in the proposition.
\end{proof}

\begin{theorem}\label{thm:hom-Weyl-isomorphism}
$A_n^k$ and $A_n^{k'}$ are isomorphic as hom-associative $K$-algebras if and only if $k$ and $k'$ contain equally many nonzero elements. 
\end{theorem}

\begin{proof}
Suppose $k$ and $k'$ contain equally many nonzero elements, or put differently, $\vert k^{-1}(0)\vert=\vert k'^{-1}(0)\vert$. Then there are bijections $\beta\colon k^{-1}(0)\to k'^{-1}(0)$ and $\beta'\colon \{i\in\N_{>0}\mid k_i\neq 0,\ i\leq n\}\to \{i\in\N_{>0}\mid k'_i\neq 0,\ i\leq n\}$. We claim that the map $\varphi\colon A_n^k\to A_n^{k'}$ defined by  

\[\varphi(x_\ell)\colonequals \begin{cases}x_{\beta(\ell)}&\text{if } \ell\in k^{-1}(0),\\
\frac{k'_{\beta'(\ell)}}{k_\ell}x_{\beta'(\ell)}&\text{else,}\end{cases}\quad\varphi(y_\ell)\colonequals \begin{cases}y_{\beta(\ell)}&\text{if } \ell\in k^{-1}(0),\\
\frac{k_\ell}{k'_{\beta'(\ell)}}y_{\beta'(\ell)}&\text{else,}\end{cases}\]
for $1\leq \ell\leq n$ is an isomorphism of hom-associative $K$-algebras. We see that $\varphi$ is the same map as in \autoref{prop:Weyl-homomorphisms} with 
\begin{align*}
p_\ell&=\begin{cases}x_{\beta(\ell)}&\text{if } \ell\in c^{-1}(0),\\
\frac{k'_{\beta'(\ell)}}{k_\ell}x_{\beta'(\ell)}&\text{else,}\end{cases} &q_\ell&= \begin{cases}y_{\beta(\ell)}&\text{if } \ell\in k^{-1}(0),\\
0&\text{else,}\end{cases}\\
f_{i\ell}&=0, &g_{i\ell}&=\begin{cases}
0&\text{if }\ell\in k^{-1}(0)\\
\frac{k_\ell}{k'_i}\delta_{i\beta'(\ell)}&\text{else.}\end{cases}
\end{align*}
Moreover, the coefficients $f_{i\ell}$ and $g_{i\ell}$ satisfy \eqref{eq:coeff1} and \eqref{eq:coeff2} in \autoref{prop:Weyl-homomorphisms}, as
\begin{align*}
\sum_{i=1}^nf_{i\ell}k_i&=0,\\
\sum_{i=1}^ng_{i\ell}k'_i&=\sum_{i=1}^n\frac{k_\ell}{k'_i}\delta_{i\beta'(\ell)}k'_i=\sum_{i=1}^nk_\ell\delta_{i\beta'(\ell)}=k_\ell.
\end{align*}
The PDEs \eqref{eq:PDE1} and \eqref{eq:PDE2} in \autoref{prop:Weyl-homomorphisms} now both read $0=0$, while \eqref{eq:PDE3} is equal to
\[
\sum_{i=1}^n\frac{\partial}{\partial x_i}(g_{i\ell}p_j)=[q_\ell,p_j]+\delta_{j\ell}.
\]
If $\ell\in k^{-1}(0)$, then $g_{i\ell}=0$, so the above PDE is equal to $0=[q_\ell,p_j]+\delta_{j\ell}=[y_{\beta(\ell)},x_{\beta(j)}]+\delta_{j\ell}=-[x_{\beta(j)},y_{\beta(\ell)}]+\delta_{j\ell}=-\delta_{\beta(j)\beta(\ell)}+\delta_{j\ell}=0$. Here, we have used that $\delta_{\beta(j)\beta(\ell)}=\delta_{j\ell}$, which is due to the fact that $\beta$ is a bijection. If $\ell\not\in k^{-1}(0)$, then the left-hand side of the above PDE reads
\[
\sum_{i=1}^n\frac{\partial}{\partial x_i}\left(\frac{k_\ell}{k'_i}\delta_{i\beta'(\ell)}\frac{k'_{\beta'(j)}}{k_j}x_{\beta'(j)}\right)=\frac{\partial}{\partial x_{\beta'(\ell)}}\left(\frac{k_\ell}{d_{\beta'(\ell)}}\frac{k'_{\beta'(j)}}{k_j}x_{\beta'(j)}\right)=\frac{k_\ell}{k_j}\delta_{\beta'(j)\beta'(\ell)},
\]
which in turn equals $\delta_{j\ell}$, the right-hand side of the same PDE, since $\beta'$ is a bijection. From \autoref{prop:Weyl-homomorphisms}, we conclude that $\varphi\in\Hom_K(A_n^k,A_n^{k'})$. Now, let us show that $\varphi$ is a bijection. By calculations similar to those above, the map $\varphi'$ defined by 
\[\varphi'(x_\ell)\colonequals \begin{cases}x_{\beta^{-1}(\ell)}&\text{if }\ell\in k^{-1}(0),\\
\frac{k_\ell}{k'_{\beta'(\ell)}}x_{\beta'^{-1}(\ell)}&\text{else,}\end{cases}\quad\varphi'(y_\ell)\colonequals \begin{cases}y_{\beta^{-1}(\ell)}&\text{if } \ell\in k^{-1}(0),\\
\frac{k'_{\beta'(\ell)}}{k_\ell}y_{\beta'^{-1}(\ell)}&\text{else,}\end{cases}\]
defines a $K$-algebra endomorphism on $A_n$. Moreover, $\varphi\circ\varphi'=\varphi'\circ\varphi=\id_{A_n}$, so $\varphi'$ is both a left inverse and a right inverse to $\varphi$, and hence $\varphi$ is a bijection. We conclude that $\varphi\colon A_n^k\to A_n^{k'}$ is an isomorphism of hom-associative $K$-algebras.

Suppose instead that $k$ and $k'$ do not contain equally many zeros, so that $\vert k^{-1}(0)\vert\neq \vert k'^{-1}(0)\vert$. Without loss of generality, we may assume $\vert k^{-1}(0)\vert > \vert  k'^{-1}(0)\vert$ (otherwise, we may switch $k$ and $k'$). If $k'=(k'_1,\ldots,k'_n)\in K^n$, then by replacing $\vert k^{-1}(0)\vert - \vert k'^{-1}(0)\vert$ of the nonzero $k'_i$s with zeroes where $1\leq i\leq n$, we may construct  $k''=(k_1'',\dots,k_n'')\in K^n$ with $\vert k''^{-1}(0)\vert =\vert k^{-1}(0)\vert$ where if $k''_j\neq0$ for some $j$ with $1\leq j\leq n$, then $k''_j=k'_j$. In particular, $k'^{-1}(0)\subsetneq k''^{-1}(0)$. Since $\vert k''^{-1}\vert = \vert k^{-1}\vert$, by the previous proof, $A_n^{k''}$ is isomorphic as a hom-associative $K$-algebra to $A_n^k$. We claim that $A_n^{k''}$ is not isomorphic as a hom-associative $K$-algebra to $A_n^{k'}$, and since being isomorphic in this sense is a transitive relation, $A_n^k$ cannot be isomorphic as a hom-associative $K$-algebra to $A_n^{k'}$. If $\delta\in\Der_K(A_n^{k'})$, then by \autoref{thm:derivations}, $\delta=\ad_p\in\InnDer_K(A_n)$ where $p=\sum_{i\not\in k'^{-1}(0)}f_iy_i+q(y_{k'^{-1}(0)},x_1,\ldots,x_n)\in A_n$, $f_i\in K$. Since $k'^{-1}(0)\subsetneq k''^{-1}(0)$, we may write 
\begin{align*}
p&=\sum_{i\not\in k'^{-1}(0)}f_iy_i+q(y_{k'^{-1}(0)},x_1,\ldots,x_n)\\
&=\sum_{i\not\in k''^{-1}(0)}f_iy_i+\sum_{i\not\in k'^{-1}(0)\backslash k''^{-1}(0)}f_iy_i+q(y_{k'^{-1}(0)},x_1,\ldots,x_n)\\
&=\sum_{i\not\in k''^{-1}(0)}f_iy_i+q'(y_{k''^{-1}(0)},x_1,\ldots,x_n)
\end{align*}
for some $q'(y_{k''^{-1}(0)},x_1,\ldots,x_n)\in A_n$. In particular, we see that $\delta\in\Der_K(A_n^{k''})$, so $\Der_K(A_n^{k'})\subseteq\Der_K(A_n^{k''})$. Again, since $k'^{-1}(0)\subsetneq k''^{-1}(0)$, there is some $\ell$ where $k'_\ell\neq0$ while $k''_\ell=0$. By \autoref{thm:derivations}, $\ad_{y_\ell^2}\in\Der_K(A_n^{k''})\backslash\Der_K(A_n^{k'})$, so $\Der_K(A_n^{k'})\subsetneq\Der_K(A_n^{k''})$. Since the set of derivations is an isomorphism invariant of nonassociative (and hence also hom-associative) algebras, $A_n^{k''}$ is not isomorphic as a hom-associative $K$-algebra to $A_n^k$.
\end{proof}

In light of \autoref{conj:Dixmier} ($\DC_n$) and \autoref{thm:hom-Weyl-isomorphism}, we end this section by formulating the following family of conjectures:

\begin{conjecture}[The Hom--Dixmier Conjecture ($\HDC_n^m$)]\label{conj:hom-Dixmier}
If $k$ and $k'$ both contain exactly $m$ nonzero elements, then any nonzero hom-associative $K$-algebra homomorphism $A_n^k\to A_n^{k'}$ is a hom-associative $K$-algebra isomorphism.  
\end{conjecture}

For any fixed $n\in\N_{>0}$, we let the conjunction of the conjectures $\HDC_n^m$ for all $m\leq n$, the \emph{Stable Hom--Dixmier Conjecture}, be denoted by $\HDC_n^{\leq n}$.

\begin{proposition}\label{prop:conj}
$\DC_n$ and $\HDC_n^{\leq n}$ are equivalent.
\end{proposition}

\begin{proof}
If $\DC_n$ holds, then by \ref{it:hom-Weyl-homomorphisms1} in \autoref{lem:hom-Weyl-homomorphisms}, $\HDC_n^m$ holds for any $m\leq n$. Hence $\DC_n$ implies $\HDC_n^{\leq n}$. If $\HDC_n^m$ holds for any $m\leq n$, in particular, $\HDC_n^0=\DC_n$ holds, so $\HDC_n^{\leq n}$ implies $\DC_n$.
\end{proof}

By the above proposition, $\HDC_n^{\leq n}$ and $\JC_n$ are therefore stably equivalent (see \autoref{subsec:iterated}) ($\HDC_1^1$ was proven true in \cite[Corollary 5]{BR20}). 

\begin{remark}
Recall from \autoref{thm:simple} that $A_n^k$ is simple for any $k\in K^n$, and that the kernel of any nonassociative (and hence also hom-associative) $K$-algebra homomorphism is an ideal. Hence any nonzero hom-associative $K$-algebra homomorphism from $A_n^k$ to $A_n^{k'}$ must be injective, this since its kernel must be the zero ideal. \autoref{conj:hom-Dixmier} is thus a conjecture about the surjectivity of such maps.
\end{remark}

\section{Multi-parameter formal hom-associative deformations}\label{sec:multi-param}
In \cite{MS10}, Makhlouf and Silvestrov introduced \emph{one-parameter formal hom-as\-so\-cia\-tive deformations} and \emph{one-parameter formal hom-Lie deformations} as hom-associative generalizations of the associative versions, developing the associated cohomology theory in low degrees. The theory was subsequently further developed in \cite{AEM11, HM19, She12}. In \cite[Proposition 9]{BR20}, the authors showed that the hom-associative first Weyl algebra over a field $K$ of characteristic zero is a one-parameter formal hom-associative deformation of the first Weyl algebra over $K$. Moreover, in \cite[Proposition 10]{BR20}, the authors showed that the above deformation induces a one-parameter formal hom-associative deformation of the corresponding Lie algebra into a hom-Lie algebra, when  the commutator is used as a hom-Lie bracket. 

In \cite{BR22}, the authors introduced \emph{multi-parameter formal hom-as\-so\-cia\-tive deformations} and \emph{multi-parameter formal hom-Lie deformations}. In \cite[Proposition 5.2]{BR22}, they showed that the hom-associative first Weyl algebra over a field $K$ of prime characteristic is a multi-parameter formal hom-associative deformation of the first Weyl algebra over $K$. Moreover, in \cite[Proposition 5.4]{BR22}, they showed that the above deformation induces a multi-parameter formal hom-associative deformation of the corresponding Lie algebra into a hom-Lie algebra, when the commutator is used as a hom-Lie bracket. 

The goal of this section is to prove that the higher-order hom-associative Weyl algebras over a field $K$ of characteristic zero are multi-parameter formal hom-associative deformations of the higher-order Weyl algebras over $K$, and that these deformations induce multi-parameter formal hom-associative deformations of the corresponding Lie algebras into a hom-Lie algebras, when the commutator is used as a hom-Lie bracket. Hence the higher-order hom-associative Weyl algebras over a field of characteristic zero share, a bit surprisingly, this property with the hom-associative first Weyl algebras over a field of prime characteristic.

We recall from \cite{BR22} the following definitions. Let $R$ be an associative, commutative, and unital ring and denote by $R\llbracket t_1,\ldots,t_m\rrbracket$ the formal power series ring over $R$ in the indeterminates $t_1,\ldots,t_m$. Suppose $A$ is a nonassociative $R$-algebra. We may now form a \emph{nonassociative power series ring} $A\llbracket t_1,\ldots,t_m\rrbracket$ over $A$ in the same indeterminates. The elements of this ring are formal sums $\sum_{i\in\N^m} a_it^i$ where $i\colonequals (i_1,\ldots,i_m)\in\N^m$ and $t^i\colonequals t_1^{i_1}\cdots t_m^{i_m}$, addition is pointwise, and multiplication is given by the usual multiplication of formal power series that we are used to from the associative setting. In particular, $A\llbracket t_1,\ldots,t_m\rrbracket$ is a nonassociative algebra over the associative, commutative, and unital ring $R\llbracket t_1,\ldots,t_m\rrbracket$. Now, it is not too hard to see that any $R$-bilinear map $\mu_i\colon A\times A\to A$ can be extended homogeneously (see \autoref{subsec:iterated}) to an $R\llbracket t_1,\ldots,t_m\rrbracket$-bilinear map $\widehat\mu_i\colon A\llbracket t_1,\ldots,t_m\rrbracket \times A\llbracket t_1,\ldots,t_m\rrbracket \to A\llbracket t_1,\ldots,t_m\rrbracket$. If instead of using juxtaposition, we denote by $\mu_0\colon A\times A\to A$ the multiplication in $A$, then $\mu_0$ is, in particular, an $R$-bilinear map that may be extended homogeneously to an $R\llbracket t_1,\ldots,t_m\rrbracket$-bilinear map $\widehat\mu_0\colon A\llbracket t_1,\ldots,t_m\rrbracket \times A\llbracket t_1,\ldots,t_m\rrbracket \to A\llbracket t_1,\ldots,t_m\rrbracket$. Similarly, any $R$-linear map $\alpha_i\colon A\to A$ may be extended homogeneously to an $R\llbracket t_1,\ldots,t_m\rrbracket$-linear map $\widehat\alpha_i\colon A\llbracket t_1,\ldots,t_m\rrbracket \to A\llbracket t_1,\ldots,t_m\rrbracket$. With these notations, we make the following definition:

\begin{definition}[Multi-parameter formal hom-associative deformation]\label{def:hom-assoc-deform} Let $R$ be an associative, commutative, and unital ring, and let $A$ be a hom-associative $R$-algebra. Let $i\colonequals (i_1,\ldots,i_m)\in\N^m$ and $t^i\colonequals t_1^{i_1}\cdots t_m^{i_m}$, and suppose $\mu_i\colon A\times A\to A$ are $R$-bilinear maps where $\mu_0$ is the multiplication in $A$, extended homogeneously to $R\llbracket t_1,\ldots,t_m\rrbracket$-bilinear maps $\widehat\mu_i\colon A\llbracket t_1,\ldots,t_m\rrbracket \times A\llbracket t_1,\ldots,t_m\rrbracket \to A\llbracket t_1,\ldots,t_m\rrbracket$. Suppose further that $\alpha_i\colon A\to A$ are $R$-linear maps where $\alpha_0$ is the twisting map on $A$, extended homogeneously to $R\llbracket t_1,\ldots,t_m\rrbracket$-linear maps $\widehat\alpha_i\colon A\llbracket t_1,\ldots,t_m\rrbracket \to A\llbracket t_1,\ldots,t_m\rrbracket$. A \emph{multi-}, or an \emph{$m$-parameter formal hom-associative deformation} of $A$, is a hom-associative algebra $A\llbracket t_1,\ldots,t_m\rrbracket$ over $R\llbracket t_1,\ldots,t_m\rrbracket$, with multiplication $\mu$ and twisting map $\alpha$ given by
\begin{equation*}
\mu=\sum_{i\in\N^m}\widehat\mu_i t^i,\quad\qquad \alpha=\sum_{i\in\N^m}\widehat\alpha_i t^i.
\end{equation*}
\end{definition}

\begin{theorem}\label{thm:hom-Weyl-deform}$A_n^k$, where $k\in K^n$ contains exactly $m$ nonzero elements, is an $m$-parameter formal hom-associative deformation of $A_n$.
\end{theorem}

\begin{proof}
Assume that $k\in K^n$ contains exactly $m$ nonzero elements. By \autoref{thm:hom-Weyl-isomorphism}, we may, without loss of generality,  assume that these are $k_1,\ldots,k_m$, so that $k=(k_1,\ldots,k_m,0,\ldots,0)$. By \autoref{prop:nilpotent-maps} and \ref{it:alpha1} in \autoref{lem:alpha-properties}, we have $\widehat\alpha_k=e^{k\frac{\partial}{\partial y}}=e^{k_1\frac{\partial}{\partial y_1}}\cdots e^{k_m\frac{\partial}{\partial y_m}}e^{0\frac{\partial}{\partial y_{m+1}}}\cdots e^{0\frac{\partial}{\partial y_n}}=e^{k_1\frac{\partial}{\partial y_1}}\cdots e^{k_m\frac{\partial}{\partial y_m}}$, where we have used that  $e^0=\id_{A_n}$. Now we see that $\widehat\alpha_k$ is indeed a formal power series, or when acting on an element of $A_n$, even a \emph{finite} formal sum, in $k_1,\ldots,k_m$ where the coefficients are $K$-linear maps $A_n\to A_n$ (compare also with the proofs of \cite[Proposition 9]{BR20} and \cite[Proposition 5.2]{BR22}). By \ref{it:alpha3} in \autoref{lem:alpha-properties}, it now follows that the multiplication $*$ in $A_n^k$ is also a formal power series in $k_1,\ldots,k_m$ where the coefficients are $K$-linear maps $A_n\to A_n$. Since $k_1,\ldots,k_m\in K$ and by \ref{it:center} in \autoref{thm:hom-Weyl-properties2}, $Z(A_n^k)=K$, we may regard $k_1,\ldots,k_m$ as indeterminates $t_1,\ldots,t_m$ in the formal power series algebra $A_n\llbracket t_1,\ldots,t_m\rrbracket$ over $K\llbracket t_1,\ldots,t_m\rrbracket$. With the multiplication and twisting map above, this algebra is hom-associative by construction.
\end{proof}

\begin{definition}[Multi-parameter formal hom-Lie deformation]\label{def:hom-Lie-deform} Let $R$ be an associative, commutative, and unital ring, and let $L$ be a hom-Lie algebra over $R$. Let $i\colonequals (i_1,\ldots,i_m)\in\N^m$ and $t^i\colonequals t_1^{i_1}\cdots t_m^{i_m}$, and suppose $[\cdot,\cdot]_i,\colon L\times L\to L$ are $R$-bilinear maps where $[\cdot,\cdot]_0$ is the hom-Lie bracket in $L$, extended homogeneously to $R\llbracket t_1,\ldots,t_m\rrbracket$-bilinear maps $\widehat{[\cdot,\cdot]}_i\colon L\llbracket t_1,\ldots,t_m\rrbracket \times L\llbracket t_1,\ldots,t_m\rrbracket \to L\llbracket t_1,\ldots,t_m\rrbracket$. Suppose further that $\alpha_i\colon L\to L$ are $R$-linear maps where $\alpha_0$ is the twisting map on $L$, extended homogeneously to $R\llbracket t_1,\ldots,t_m\rrbracket$-linear maps $\widehat\alpha_i\colon L\llbracket t_1,\ldots,t_m\rrbracket \to L\llbracket t_1,\ldots,t_m\rrbracket$. A \emph{multi-}, or an \emph{$m$-parameter formal hom-Lie deformation} of $L$, is a hom-Lie algebra $L\llbracket t_1,\ldots,t_m\rrbracket$ over $R\llbracket t_1,\ldots,t_m\rrbracket$, with hom-Lie bracket $[\cdot,\cdot]$ and twisting map $\alpha$ given by
\begin{equation*}
[\cdot,\cdot]=\sum_{i\in\N^m}\widehat{[\cdot,\cdot]}_i t^i,\quad\qquad \alpha=\sum_{i\in\N^m}\widehat\alpha_i t^i.
\end{equation*}
\end{definition}

By using \ref{it:alpha3} in \autoref{lem:alpha-properties}, we see that the commutator in $A_n^k$ equals $e^{k\frac{\partial}{\partial y}} [\cdot,\cdot]$, where $[\cdot,\cdot]$ is the commutator in $A_n$. In particular, the commutator in $A_n^k$ is a formal power series in the nonzero elements of $k$ where the coefficients are $K$-bilinear maps $A_n\times A_n\to A_n$. By \autoref{prop:hom-Lie-commutator} and \autoref{thm:hom-Weyl-deform} (compare also with the proofs of \cite[Proposition 10]{BR20} and \cite[Proposition 5.4]{BR22}), we thus have the following result:

\begin{corollary}\label{cor:hom-Weyl-Lie-deform}The deformation of $A_n$ into $A_n^k$ induces an $m$-parameter formal hom-Lie deformation of the Lie algebra of $A_n$ into the hom-Lie algebra of $A_n^k$ when the commutator is used as a hom-Lie bracket.
\end{corollary}

\section*{Acknowledgments}
The author would like to thank the anonymous referee for valuable comments that helped improve the manuscript.


\begin{thebibliography}{99}
\bibitem{AEM11}
F.~Ammar, Z.~Ejbehi, and A.~Makhlouf,
\emph{Cohomology and deformations of hom-algebras},
J. Lie Theory {\bf 21}(4) (2011), pp. 813--836.

\bibitem{AB26}
M.~Aryapoor and P.~B{\"a}ck,
\emph{Hilbert's basis theorem for generalized nonassociative Ore extensions},
\texttt{arXiv:2512.02597}.

\bibitem{AB25}
M.~Aryapoor and P.~B{\"a}ck,
\emph{Flipped non-associative polynomial rings and the Cayley--Dickson construction},
J. Algebra {\bf 662} (2025), pp. 482--501.

\bibitem{Bac22}
P.~B{\"a}ck,
\emph{On hom-associative Ore extensions},
PhD dissertation, M{\"a}lardalen University, V{\"a}ster{\aa}s, 2022.

\bibitem{BLOR24}
P.~B{\"a}ck, P. Lundstr{\"o}m, J.~{\"O}inert, and J. Richter,
\emph{Ore extensions of abelian groups with operators},
J. Algebra {\bf 686} (2026), pp. 176--194.

\bibitem{BR23}
P.~B{\"ack} and J.~Richter,
\emph{Hilbert’s basis theorem for non-associative and hom-associative Ore extensions},
Algebr. Represent. Th. {\bf 26} (2023), pp. 1051--1065.

\bibitem{BR24}
P.~B{\"ack} and J.~Richter,
\emph{Ideals in hom-associative Weyl algebras},
Int. Electron. J. Algebra {\bf 38} (2025), 74--81.

\bibitem{BR24b}
P.~B{\"ack} and J.~Richter,
\emph{Non-associative versions of Hilbert’s basis theorem},
Colloq. Math. {\bf 176} (2024), pp. 135--145.

\bibitem{BR20}
P.~B{\"a}ck and J.~Richter,
\emph{On the hom-associative Weyl algebras},
J. Pure Appl. Algebra {\bf 224}(9) (2020).

\bibitem{BR22}
P.~B{\"ack} and J.~Richter,
\emph{The hom-associative Weyl algebras in prime characteristic},
Int. Electron. J. Algebra {\bf 31} (2022), pp. 203--229.

\bibitem{BRS18}
P.~B{\"ack}, J.~Richter, and S.~Silvestrov,
\emph{Hom-associative Ore extensions and weak unitalizations},
Int. Electron. J. Algebra {\bf 24} (2018), pp. 174--194.

\bibitem{Cou97}
S~.C~.~Coutinho,
\emph{The many avatars of a simple algebra},
Amer. Math. Monthly {\bf 104}(7) (1997), pp. 593--604.

\bibitem{Dix68}
J.~Dixmier,
\emph{Sur les alg\`{e}bres de Weyl},
Bull. Soc. Math. France {\bf 96} (1968), pp. 209--242.

\bibitem{FG09}
Y.~Fr{\'e}gier and A.~Gohr,
\emph{On unitality conditions for hom-associative algebras},
\texttt{arXiv:0904.4874}.

\bibitem{Ger64}
M.~Gerstenhaber,
\emph{On the deformation of rings and algebras},
Ann. Math. {\bf 79}(1) (1964), pp. 59--103.

\bibitem{GG14}
M.~Gerstenhaber and A.~Giaquinto,
\emph{On the cohomology of the Weyl algebra, the quantum plane, and the $q$-Weyl algebra},
J. Pure Appl. Algebra {\bf 218}(5) (2014), pp. 879-887.

\bibitem{GW04}
K.~R.~Goodearl and R.~B.~Warfield,
\emph{An introduction to noncommutative Noetherian rings},
2nd. ed., Cambridge University Press, Cambridge U.K., 2004.

\bibitem{HLS06}
J.~T.~Hartwig, D.~Larsson, and S.~D.~Silvestrov,
\emph{Deformations of Lie algebras using $\sigma$-derivations},
J. Algebra {\bf 295}(2) (2006), pp. 314--361.

\bibitem{HM19}
B.~Hurle and A.~Makhlouf,
\emph{$\alpha$-type Hochschild cohomology of hom-associative algebras and bialgebras},
J. Korean Math. Soc., {\bf 56}(6) (2019), pp. 1655--1687.

\bibitem{Hir37}
K.~A.~Hirsch,
\emph{A note on non-commutative polynomials},
J. London Math. Soc. {\bf 12}(4) (1937), pp. 264--266.

\bibitem{KBK07}
A.~Kanel-Belov and M.~Kontsevich,
\emph{The Jacobian conjecture is stably equivalent to the Dixmier conjecture},
Mosc. Math. J. \textbf{7}(2) (2007), pp. 209--218.

\bibitem{Lam01}
T.~Y.~Lam,
\emph{A first course in noncommutative rings},
2nd ed. Springer, New York, 2001.

\bibitem{Lit33}
D.~E.~Littlewood,
\emph{On the classification of algebras},
Proc. London Math. Soc. {\bf 35}(1) (1933), pp. 200--240.

\bibitem{MS08}
A.~Makhlouf and S.~D.~Silvestrov,
\emph{Hom-algebra structures}, 
J. Gen. Lie Theory Appl. {\bf 2}(2) (2008), pp. 51--64.

\bibitem{MS10}
A.~Makhlouf and S.~Silvestrov,
\emph{Notes on 1-parameter formal deformations of hom-associative and hom-Lie algebras}, 
Forum Math. {\bf 22}(4) (2010), pp. 715--739.

\bibitem{MR01}
J.~C.~McConnell and J.~C.~Robson, with the cooperation of L.~W.~Small, 
\emph{Noncommutative Noetherian rings}, 
Rev. ed., Amer. Math. Soc., Providence, R.I., 2001.

\bibitem{NOR18}
P.~Nystedt, J.~{\"O}inert, and J.~Richter,
\emph{Non-associative Ore extensions},
Isr. J. Math. {\bf 224}(1) (2018), pp. 263--292.

\bibitem{NOR19}
P.~Nystedt, J. {\"O}inert, and J. Richter, 
\emph{Simplicity of Ore monoid rings}, 
J. Algebra {\bf 530} (2019), pp. 69--85.

\bibitem{Ore33}
O.~Ore,
\emph{Theory of non-commutative polynomials}, 
Ann. Math. {\bf 34}(3) (1933), pp. 480--508.

\bibitem{Row88}
L.~H.~Rowen, 
\emph{Ring theory. Vol. I},
Academic, Boston, 1988.

\bibitem{She12}
Y.~Sheng,
\emph{Representations of hom-Lie algebras},
Algebr. Represent. Th. {\bf 15} (2012), pp. 1081--1098.

\bibitem{Sri61}
R.~Sridharan,
\emph{Filtered algebras and representations of Lie algebras},
Trans. Amer. Math. Soc. {\bf 100}(3) (1961), pp. 530--550.

\bibitem{Tsu05}
Y.~Tsuchimoto,
\emph{Endomorphisms of Weyl algebra and $p$-curvatures}, 
Osaka J. Math. {\bf 42}(2) (2005), pp. 435--452.

\bibitem{Yau09}
D.~Yau,
\emph{Hom-algebras and homology},
J. Lie Theory {\bf 19}(2) (2009), pp. 409--421.

\bibitem{Zhe24}
A.~Zheglov,
\emph{The conjecture of Dixmier for the first Weyl algebra is true},
\texttt{arXiv:2410.06959}.

\end{thebibliography}
\end{document}